\documentclass[12pt,leqno,draft]{article}
\usepackage{amsfonts}
\usepackage{amsmath, amsthm, amsfonts, amssymb, color}
\usepackage{mathrsfs}
\usepackage{color}
\usepackage{xcolor}

\newtheorem{thm}{Theorem}[section]

\newtheorem{lem}[thm]{Lemma}

\newtheorem{exa}[thm]{Example}
\newtheorem{rem}[thm]{Remark}
\theoremstyle{definition}

\newcommand{\scr}[1]{\mathscr #1}
\definecolor{wco}{rgb}{0.5,0.2,0.3}

\numberwithin{equation}{section} \theoremstyle{remark}

\newcommand{\ua}{\uparrow}

\title{{\bf
Log-Sobolev Inequalities and Exponential Ergodicity for  Non-degenerate and Degenerate  McKean-Vlasov SDEs}\footnote{
Xing Huang  is Supported in
 part by  National Key R\&D Program of China (No. 2022YFA1006000) and NNSFC (12271398). Eva Kopfer is supported by the German Research Foundation
 		through the Hausdorff Center for Mathematics and the Collaborative Research Center 1060. Panpan Ren is supported by NNSFC (12301180), RGC (21301925) and Research Center for Nonlinear Analysis at The Hong Kong Polytechnic University. } }
\author{
{\bf   Xing Huang$^{(a)}$,  Eva Kopfer$^{(b)}$,  Panpan Ren$^{(c)}$}\\
\footnotesize{$^{a)}$  Center for Applied Mathematics, Tianjin University, Tianjin 300072, China}\\
\footnotesize{$^{b)}$ Institut f\"ur Angewandte Mathematik, Universit\"at Bonn, Endenicher Allee 60,  Bonn,  Germany }\\
\footnotesize{$^{c)}$ Department of Mathematics, City University of  Hong Kong, Tat Chee Avenue, Hong Kong,  China }\\
\footnotesize{ xinghuang@tju.edu.cn;  eva.kopfer@iam.uni-bonn.de;  panparen@cityu.edu.hk}}
\begin{document}
\allowdisplaybreaks
\def\R{\mathbb R}  \def\ff{\frac} \def\ss{\sqrt} \def\B{\mathbf
B} \def\W{\mathbb W}
\def\N{\mathbb N} \def\kk{\kappa} \def\m{{\bf m}}
\def\ee{\varepsilon}\def\ddd{D^*}
\def\dd{\delta} \def\DD{\Delta} \def\vv{\varepsilon} \def\rr{\rho}
\def\<{\langle} \def\>{\rangle} \def\GG{\Gamma} \def\gg{\gamma}
  \def\nn{\nabla} \def\pp{\partial} \def\E{\mathbb E}
\def\d{\text{\rm{d}}} \def\bb{\beta} \def\aa{\alpha} \def\D{\scr D}
  \def\si{\sigma} \def\ess{\text{\rm{ess}}}
\def\beg{\begin} \def\beq{\begin{equation}}  \def\F{\scr F}
\def\Ric{\text{\rm{Ric}}} \def\Hess{\text{\rm{Hess}}}
\def\e{\text{\rm{e}}} \def\ua{\underline a} \def\OO{\Omega}  \def\oo{\omega}
 \def\tt{\tilde} \def\Ric{\text{\rm{Ric}}}
\def\cut{\text{\rm{cut}}} \def\P{\mathbb P} \def\ifn{I_n(f^{\bigotimes n})}
\def\C{\scr C}      \def\aaa{\mathbf{r}}     \def\r{r}
\def\gap{\text{\rm{gap}}} \def\prr{\pi_{{\bf m},\varrho}}  \def\r{\mathbf r}
\def\Z{\mathbb Z} \def\vrr{\varrho} \def\ll{\lambda}
\def\L{\scr L}\def\Tt{\tt} \def\TT{\tt}\def\II{\mathbb I}
\def\i{{\rm in}}\def\Sect{{\rm Sect}}  \def\H{\mathbb H}
\def\M{\scr M}\def\Q{\mathbb Q} \def\texto{\text{o}} \def\LL{\Lambda}
\def\Rank{{\rm Rank}} \def\B{\scr B} \def\i{{\rm i}} \def\HR{\hat{\R}^d}
\def\to{\rightarrow}\def\l{\ell}\def\iint{\int}
\def\EE{\scr E}\def\Cut{{\rm Cut}}
\def\A{\scr A} \def\Lip{{\rm Lip}}
\def\BB{\scr B}\def\Ent{{\rm Ent}}\def\L{\scr L}
\def\R{\mathbb R}  \def\ff{\frac} \def\ss{\sqrt} \def\B{\mathbf
B}
\def\N{\mathbb N} \def\kk{\kappa} \def\m{{\bf m}}
\def\dd{\delta} \def\DD{\Delta} \def\vv{\varepsilon} \def\rr{\rho}
\def\<{\langle} \def\>{\rangle} \def\GG{\Gamma} \def\gg{\gamma}
  \def\nn{\nabla} \def\pp{\partial} \def\E{\mathbb E}
\def\d{\text{\rm{d}}} \def\bb{\beta} \def\aa{\alpha} \def\D{\scr D}
  \def\si{\sigma} \def\ess{\text{\rm{ess}}}
\def\beg{\begin} \def\beq{\begin{equation}}  \def\F{\scr F}
\def\Ric{\text{\rm{Ric}}} \def\Hess{\text{\rm{Hess}}}
\def\e{\text{\rm{e}}} \def\ua{\underline a} \def\OO{\Omega}  \def\oo{\omega}
 \def\tt{\tilde} \def\Ric{\text{\rm{Ric}}}
\def\cut{\text{\rm{cut}}} \def\P{\mathbb P} \def\ifn{I_n(f^{\bigotimes n})}
\def\C{\scr C}      \def\aaa{\mathbf{r}}     \def\r{r}
\def\gap{\text{\rm{gap}}} \def\prr{\pi_{{\bf m},\varrho}}  \def\r{\mathbf r}
\def\Z{\mathbb Z} \def\vrr{\varrho} \def\ll{\lambda}
\def\L{\scr L}\def\Tt{\tt} \def\TT{\tt}\def\II{\mathbb I}
\def\i{{\rm in}}\def\Sect{{\rm Sect}}  \def\H{\mathbb H}
\def\M{\scr M}\def\Q{\mathbb Q} \def\texto{\text{o}} \def\LL{\Lambda}
\def\Rank{{\rm Rank}} \def\B{\scr B} \def\i{{\rm i}} \def\HR{\hat{\R}^d}
\def\to{\rightarrow}\def\l{\ell}\def\BB{\mathbb B}
\def\8{\infty}\def\I{1}\def\U{\scr U} \def\n{{\mathbf n}}\def\v{V}
\maketitle

\begin{abstract}
The exponential ergodicity of partially dissipative McKean-Vlasov SDEs in the \(L^1\)-Wasserstein distance has been extensively studied using asymptotic reflection coupling. However, the reflection coupling method is not applicable for the exponential ergodicity in $L^2$-Wasserstein distance and relative entropy. In this paper, we first establish uniform log-Sobolev inequalities (in the frozen measure variable with bounded second moments) for the invariant probability measure of the corresponding SDEs with frozen distribution. Second, for the McKean-Vlasov SDEs, we combine the log-Harnack inequality and Talagrand's inequality to derive exponential ergodicity in both
$L^2$-Wasserstein distance and relative entropy. Furthermore, we extend these main results to the case of degenerate diffusion.
 \end{abstract}

\noindent
 AMS subject Classification:\  60H10, 60K35, 82C22.   \\
\noindent
 Keywords: Log-Sobolev inequality, Dimension-free Harnack inequality, Talagrand inequality,  McKean-Vlasov SDEs, exponential ergodicity, $L^2$-Wasserstein distance
 \vskip 2cm

 \tableofcontents
\section{Introduction}
\subsection{Backgroud on McKean-Vlasov SDEs}

The McKean-Vlasov SDE, first introduced by McKean \cite{McKean}, was developed to study the nonlinear Fokker-Planck equation. The idea goes back to Kac \cite{Kac,Kac1}, who characterized nonlinear PDEs in Vlasov's kinetic theory. The main tool is the so-called ``propagation of chaos" which builds a corner stone in mean field interacting particle systems, see e.g. \cite{CD,meleard1996,SZ}.
In general, the McKean-Vlasov SDE is given by
\begin{align}\label{E1}
\d X_t=b_t(X_t,\L_{X_t})\d t+\sigma_t(X_t,\L_{X_t})\d W_t,
\end{align}
where $\{W_t\}_{t\geq 0}$ is an $n$-dimensional Brownian motion on some complete filtered probability space $(\Omega, \scr F, (\scr F_t)_{t\geq 0},\P)$ and $\L_{X_t}$ is the distribution of $X_t$. The coefficients
 $b:[0,\infty)\times \R^d\times\scr P(\R^d)\to\R^d$, $\sigma:[0,\infty)\times \R^d\times\scr P(\R^d)\to\R^d\otimes\R^{n}$ are assumed to be measurable and bounded on bounded sets, where
 $\scr P(\R^d)$ is the set of all probability measures on $\R^d$ equipped with the weak topology. We will consider the solution of \eqref{E1} in a subspace of $\scr P(\R^d)$. To this end, for $p\geq 1$, let
$$\scr P_p(\R^d):=\big\{\mu\in \scr P(\R^d): \|\mu\|_p:=\mu(|\cdot|^p)^{\ff 1 p}<\infty\big\},$$
which is a Polish space under the  $L^p$-Wasserstein distance
$$\W_p(\mu,\nu)= \inf_{\pi\in \C(\mu,\nu)} \bigg(\int_{\R^d\times\R^d} |x-y|^p \pi(\d x,\d y)\bigg)^{\ff 1 {p}},\ \ \mu,\nu\in\scr P_p(\R^d), $$ where $\C(\mu,\nu)$ is the set of all couplings of $\mu$ and $\nu$.

When \eqref{E1} is well-posed in $\scr P_p(\R^d)$, for given $\gamma\in\scr P_p(\R^d)$ we denote $P_t^*\gamma=\L_{X_t}$ with $\L_{X_0}=\gamma$.

Recall that the relative entropy between $\mu,\nu\in\scr P(\R^d)$ is defined as
$$\mathrm{Ent}(\nu|\mu)=\left\{
  \begin{array}{ll}
    \nu(\log(\frac{\d \nu}{\d \mu})), & \hbox{$\nu$ \text{is absolutely continuous w.r.t.} $\mu$;} \\
    \infty, & \hbox{otherwise.}
  \end{array}
\right.$$

In this paper, we aim to advance the study of functional inequalities for McKean-Vlasov SDEs and their associated equations with frozen distribution and to apply these results to analyze exponential ergodicity in both the $L^2$-Wasserstein distance and relative entropy.  Our approach builds on suitable Wang's Harnack inequality and Wang's observation in \cite{W2020}(and
 \cite{Villani} in the kinetic case), which states that exponential ergodicity in relative entropy can be deduced from log-Sobolev inequalities.
To place our results into context, we briefly review key results on McKean-Vlasov SDEs, focusing on  log-Sobolev inequalities and exponential ergodicity under various metrics.

\subsection{Related works}

\subsubsection{Log-Sobolev inequalities for McKean-Vlasov SDEs}
We begin by reviewing existing results on log-Sobolev inequalities.
The study of log-Sobolev inequalities originated in the works of Leonard Gross \cite{Gross}, who introduced this for Gaussian measures. The seminal work by Bakry and \'Emery \cite{BE}  links the log-Sobolev inequality to lower curvature bounds of diffusion processes via the so-called $\Gamma$-calculus introduced therein. Since then, log-Sobolev inequalities are an integral part in the study of SDEs and PDEs, e.g. in \cite{Carlen, Diaconis,GLWZ, Ledoux,M,OV,Wanglog}.

Recently, the log-Sobolev inequality for McKean-Vlasov SDEs also attracts much attention.
For SDEs of the form
\[
\d X_t = b(X_t, \scr L_{X_t})\, \d t + \sigma \, \d W_t,
\]
\cite[Corollary 2.7]{LL} establishes a uniform-in-time log-Sobolev inequality. They assume that $\sigma = I_{d\times d}$,
\[
b(x, \mu) = -\nabla U^1(x) + \int_{\R^d} -\nabla U^2(x-y)\, \mu(\d y),
\]
$U^1, U^2$ are twice continuously differentiable, $\nabla^2(U^1 + U^2) \ge \alpha I_{d\times d}$ for some $\alpha>0$, $\nabla U^2$ is even and bounded or Lipschitz, and $\nabla U^i$ satisfies certain monotonicity and growth conditions. The result holds provided the Lipschitz constant of $\nabla U^2$ is smaller than $\alpha$ and the initial law $\scr{L}_{X_0}$ satisfies a log-Sobolev inequality.

On the torus $\mathbb{T}^d$, similar results are obtained under convolution-type interactions. In \cite[Corollary 2.9]{LL}, when $b(x, \mu) = \int_{\mathbb{T}^d} K(x-y)\, \mu(\d y)$ for a Lipschitz kernel $K$, a uniform log-Sobolev inequality holds if $\|\mathrm{div}\, K\|_\infty$ is sufficiently small and the initial density has bounded derivatives and is uniformly bounded away from zero and infinity. \cite[Theorem 2.5]{GLM} extends this by allowing divergence-free interaction kernels of the form $K = \mathrm{div}\, V$, again under smoothness and boundedness conditions on the initial density, using the Holley-Stroock perturbation lemma and the log-Sobolev inequality of the uniform distribution on $\mathbb{T}^d$.

For time-inhomogeneous SDEs, \cite[Theorem 1.3]{MRW} shows that a uniform-in-time log-Sobolev inequality holds when the drift $b_t$ satisfies a local monotonicity condition inside a ball, a strong dissipativity condition outside, and a mild growth condition, provided the diffusion coefficient $\sigma^2$ is large enough. This result is then applied to McKean-Vlasov SDEs in \cite[Proposition 3.9]{MRW} with drifts of the form $-\nabla U - \int \nabla_x W(\cdot, y)\, \scr{L}_{X_t}(\d y)$, where $W$ is split into a bounded component and a Lipschitz component with sufficiently small Lipschitz constant. Under these assumptions, the uniform-in-time log-Sobolev inequality holds if it holds for the initial law.

In this paper, Theorem \ref{Poi3} improves upon existing results by establishing a explicit log-Sobolev inequality for time-inhomogeneous semigroup with frozen distribution and McKean-Vlasov SDEs under mild assumptions: the drift is only locally bounded and continuous in the spatial variable, and the diffusion may depend on the measure. This is different from prior works that require $C^2$ regularity on the coefficients, see for instance \cite[Theorem 4.2]{CG}. When the coefficients satisfy partially dissipative condition, both in the non-degenerate and degenerate case, the invariant probability measure for the time-homogeneous SDEs with frozen distribution also satisfies log-Sobolev inequality even if the invariant probability measure does not have explicit representation. The key point also lies in that this log-Sobolev inequality is uniform in the froze measure with bounded second moment.

\subsubsection{Exponential ergodicity for McKean-Vlasov SDEs}
In the case $b_t=b,\sigma_t=I_{d\times d}$, exponential ergodicity of a system guarantees the existence of an invariant probability measure with exponential rate of convergence measured in some metric. For the McKean-Vlasov SDE, the $L^2$-Wasserstein distance and the relative entropy are frequently used. Provided that the coefficient is uniformly dissipative, i.e.
\begin{align*}
&\langle b(x,\mu)-b(y,\tilde{\mu}),x-y\rangle \leq -\alpha|x-y|^2+\beta\W_2(\mu,\tilde{\mu})^2,\ \ x,y\in\R^d,\mu,\tilde{\mu}\in\scr P_2(\R^d),
\end{align*}
for some $\alpha>\beta>0$, exponential ergodicity holds in the $L^2$-Wasserstein distance as well as in relative entropy, see \cite{HW24,  RW, FYW1} and references therein.
In the partially dissipative case, i.e. dissipation only occurs at long distances, exponential ergodicity has been established in the $L^1$-Wasserstein distance, see \cite{EGZ,HLM,LWZ,WR2024} and references therein, using asymptotic reflection couplings. Additionally, in the symmetric case, i.e. the invariant probability measure has explicit representation,  exponential ergodicity in the $L^2$-Wasserstein distance has been proven under the condition that the drift is bounded in the distribution variable and the diffusion is non-degenerate and only depends on the distribution variable, see \cite{ZSQ}. For recent progresses on long-time behaviour of McKean-Vlasov SDEs, we refer to \cite{WSB} on exponential ergodicity in total variation distance, \cite{GM} concerning exponential ergodicity in $L^2$-Wasserstein distance and mean field entropy for kinetic McKean-Vlasov SDEs under non-convex assumption and \cite{MS} regarding non-asymptotic entropic bounds of stochastic algorithms. Also relevant is \cite[Theorem 1.2]{PM}, which proves $L^2$-Wasserstein contraction for distribution-independent SDEs. For this result, the drift is partially dissipative,  and $\sigma=\sigma_0I_{d\times d}$ for large enough $\sigma_0$. The author's method relies on constructing a new distance that is equivalent to the squared Euclidean norm and then employing a synchronous coupling argument.

%


This paper establishes exponential ergodicity in the $L^2$-Wasserstein distance and relative entropy for McKean-Vlasov SDEs with partially dissipative drifts and weak enough interaction, which may be unbounded in the distribution variable and encompass both non-degenerate and degenerate diffusions. This finding extends the scope of \cite{RW} to partially dissipative settings and generalizes \cite{Baudoin,Villani} to non-gradient-type drifts. The proof relies on uniform log-Sobolev inequalities for the frozen system, derived using Wang's Harnack inequality, the semigroup log-Sobolev inequality, and the weak Poincar\'e inequality. Most importantly, Wang's Harnack inequality allows us to skip the explicit form of the invariant probability measure.

\paragraph{Outline}
The remainder of this paper is organized as follows. In Section 2, we study the non-degenerate case and we prove the log-Sobolev inequality for semigroup $P_{s,t}^\mu$ defined in \eqref{Ptf} and $P_t^\ast$  in Theorem \ref{Poi3} under the monotonicity assumption ({\bf{A}}). In the partially dissipative setting ({\bf H}), we derive the log-Sobolev inequality for the invariant probability measure associated to the time-homogeneous SDEs with frozen distribution (Theorem \ref{logta0}). As an application, we deduce the exponential ergodicity in $L^2$-Wasserstein distance and relative entropy for McKean-Vlasov SDEs in Theorem \ref{Erg}.  In Section 3, we concentrate on the degenerate case. For partially dissipative drifts, we derive the log-Sobolev inequality for the invariant probability measure associated to the time-homogeneous SDEs with frozen distribution in Theorem \ref{logta}.
As an application, we obtain exponential ergodicity for McKean-Vlasov SDEs in $L^2$-Wasserstein distance and relative entropy  for both uniformly dissipative (Theorem \ref{cty}) and partially dissipative drifts (Theorem \ref{thm: erg deg}).   In Section 4,  we provide the well-posedness of McKean-Vlasov SDEs and SDEs with frozen distribution  under the condition  ({\bf A}) in Theorem \ref{wel}.  The approximation approaches used to prove Theorem \ref{Poi3} also have been provided.

\section{Non-degenerate case}

In what follows, we will first provide the log-Sobolev inequality for the time-marginal distribution of the decoupled SDE, given by
\begin{align}\label{E2}
	\d X_{s,t}^{\mu}=b_t(X_{s,t}^{\mu},\mu_t)\d t+\sigma_t(X_{s,t}^{\mu},\mu_t)\d W_t,\ \ t\geq s\geq 0,
\end{align}
for any $\mu_\cdot\in C([0,\infty),\scr P_2(\R^d))$, i.e. we will prove the log-Sobolev inequality for
the associated time-inhomogeneous semigroup defined by
\begin{align}\label{Ptf}
	P_{s,t}^\mu f(x)=\E f(X_{s,t}^{\mu,x}),\ \ t\geq s\geq 0,f\in\scr B_b(\R^d),
\end{align}
where $X_{s,t}^{\mu,x}$ solve \eqref{E2} with $X_{s,s}^{\mu,x}=x\in\R^d$.

\begin{enumerate}
\item[{\bf (A)}] $b:[0,\infty)\times \R^d\times \scr P_2(\R^d)\mapsto\R^d$ is bounded on bounded sets.  For any $t\geq 0$ and $\gamma\in\scr P_2(\R^d)$, $b_t(\cdot,\gamma)$ and $\sigma_t(\cdot,\gamma)$ are continuous in $\R^d$. There exist  constants $K_1,K_I>0$ such that for all $t\geq 0$, $x,y\in\R^d$ and $\gamma,\tilde\gamma\in\scr P_2(\R^d)$,
\begin{align}\label{Mon}&\nonumber2\<b_t(x,\gamma)-b_t(y,\tilde\gamma),x-y\>+\|\sigma_t(x,\gamma)-\sigma_t(y,\tilde\gamma)\|_{HS}^2\\
&\leq K_1|x-y|^2+K_I\W_2(\gamma,\tilde\gamma)^2,
\end{align}
and
$$\|\sigma_t(x,\gamma)\|_{HS}^2\leq K_1(1+|x|^2+\gamma(|\cdot|^2)).$$
\end{enumerate}

Note that in {\bf(A)}, the local Lipschitz continuity of $b$ and $\sigma$ in spatial variable is not required.  Under {\bf(A)}, \eqref{E1} is well-posed in $\scr P_2(\R^d)$ and recall that $P_t^\ast\gamma$ is the distribution of the solution to  \eqref{E1} with initial distribution $\gamma\in\scr P_2(\R^d)$. Moreover,  for any $\mu_\cdot\in C([0,\infty),\scr P_2(\R^d))$, \eqref{E2} has a unique solution $X_{s,t}^{\mu,x}$ with initial value $x\in\R^d$ and associated time-inhomogenous semigroup $P_{s,t}^\mu$.  More details please refer to Theorem \ref{wel} in appendix.


\subsection{Log-Sobolev inequality for $P_{s,t}^\mu$ and $P_t^\ast\mu_0$}
In this subsection, we will establish a log-Sobolev inequality for the semigroup $P_{s,t}^\mu$ defined in \eqref{Ptf}. When $b,\sigma$ are distribution free, a related result was previously obtained in \cite[Theorem 4.2]{CG} under condition \eqref{Mon} and suitable $C^2$-regularity on the coefficients, using classical stochastic methods. Here, we derive this inequality under slightly different assumptions by employing the $\Gamma$-calculus developed in \cite{BE}, together with an approximation argument.

 To this end, for any $\mu_\cdot\in C([0,\infty),\scr P_2(\R^d))$, we define
$$b_t^{\mu}(x)=b_t(x,\mu_t), \ \ \sigma_t^{\mu}(x)=\sigma_t(x,\mu_t).$$
Let $C_b^\infty(\R^d):=\{f\in C^\infty(\R^d): \|\nabla^i f\|_\infty<\infty,i\geq 1\}$,
 and $C_0^\infty(\R^d)$ denote the set of all smooth functions on $\R^d$ with compact support. The associated diffusion operator is then given by
\begin{align}\label{ger}
	\scr L_t^\mu f=\frac{1}{2}\mathrm{tr}(\sigma_t^\mu(\sigma_t^\mu)^\ast\nabla ^2f)+\nabla_{b_t^\mu} f,\ \ f\in C_b^\infty(\R^d),
\end{align}
which gives rise to the $\Gamma$-operators
\begin{align*}
	\Gamma^{t,\mu}_1(f,g)=&\scr L_t^\mu(fg)-g\scr L_t^\mu f-f\scr L_t^\mu g=\<(\sigma_t^\mu)^\ast\nabla f,(\sigma_t^\mu)^\ast\nabla g\>,\\
\Gamma^{t,\mu}_2(f,g)=&\scr L_t^\mu(\Gamma_1^{t,\mu}(f,g))-\Gamma_1^{t,\mu}(\scr L_t^\mu f,g)-\Gamma_1^{t,\mu}(\scr L_t^\mu g,f).
\end{align*}
For simplicity, we denote $\Gamma^{t,\mu}_i(f)=\Gamma^{t,\mu}_i(f,f), i=1,2$.





\begin{thm}\label{Poi3} Assume ${\bf (A)}$ and suppose that $\sigma_t(x,\gamma)=\sigma_t(\gamma)$ satisfies for some constant $\delta_1>0$,
\begin{align}\label{ups}
\sigma_t(\gamma)\sigma_t(\gamma)^\ast\leq \delta_1 I_{d\times d},\ \ \gamma\in\scr P_2(\R^d), t\geq 0.
\end{align}
Then, $P_{s,t}^\mu$ satisfies the log-Sobolev inequality: for any  $\mu_\cdot\in C([0,\infty),\scr P_2(\R^d))$ and $ f\in C_b^\infty(\R^d)$, $f>0$,
\begin{align}\label{til} P_{s,t}^\mu(f\log f)-P_{s,t}^\mu f\log P_{s,t}^\mu f \leq \frac{2\delta_1 }{K_1}(\e^{K_1(t-s)}-1)P_{s,t}^\mu\big( |\nabla   f^{\frac{1}{2}}|^2 \big),\ \ 0\leq s\leq t.
\end{align}
If moreover $\gamma\in\scr P_2(\R^d)$ satisfies the log-Sobolev inequality: for some $L_0\geq 0$,
\begin{align}\label{ilo}\gamma(f\log f)\leq \gamma(f)\log \gamma(f)+2L_0\gamma\big( |\nabla   f^{\frac{1}{2}}|^2 \big),\ \  f\in C_b^\infty(\R^d), f>0,
\end{align}
then $P_t^\ast\gamma$ satisfies the log-Sobolev inequality:
\begin{equation}\label{rlo}
\begin{split}
(P_t^\ast\gamma)&(f\log f)-(P_t^\ast\gamma)(f)\log (P_t^\ast\gamma)(f)\\
 &\leq 2\Big(L_0\e^{K_1t}+ \frac{\delta_1}{K_1}(\e^{K_1t}-1)\Big) (P_{t}^{\ast}\gamma)\big( |\nabla   f^{\frac{1}{2}}|^2 \big),\ \ f\in C_b^\infty(\R^d), f>0,t\geq 0.
\end{split}
\end{equation}
\end{thm}
\begin{rem}
 In \cite[Proposition 1.1]{MRW}, the authors derive the log-Sobolev inequality for the semigroup $P_{s,t}$ associated to
$$\d X_t=b_t(X_t)\d t+\d W_t$$
under the assumption
$$\<b_t(x)-b_t(y),x-y\>\leq L|x-y|^2.$$
Then, if the distribution $\L_{X_0}$ of the initial condition $X_0$ satisfies a log-Sobolev inequality,  the distribution $\L_{X_t}$ at any $t>0$, also satisfies a  log-Sobolev inequality (with a time-dependent constant).
\end{rem}

\begin{rem} If $\sigma=I_{d\times d}$ and $b$ is of gradient-type, then the
$L^1$-gradient estimate \eqref{cdy}, the log-Sobolev inequality for $P_{s,t}^\mu$ and the Bakry--\'Emery curvature condition
$$\Gamma^{t,\mu}_2(f)\geq \rho\Gamma_1^{t,\mu}(f)=\rho|\nabla f|^2$$
for some constant $\rho\in\R$ are equivalent,  see e.g. \cite[Proposition 3.12]{M}.
\end{rem}

\begin{rem}
One can also refer to \cite[Theorem 4.1]{CM} for the log-Sobolev inequality of time-inhomogeneous semigroup and \cite[Lemma 5.2]{RSW} for the reflecting case under the Bakry-\'{E}mery curvature condition. Moreover, \eqref{cdy} is equivalent to \cite[(10)]{CM} if we assume $\sigma\sigma^\ast\geq \delta_2$ for some $\delta_2>0$.
\end{rem}

\begin{proof} For any  $\mu_\cdot\in C([0,\infty),\scr P_2(\R^d))$, recall that $X_{s,t}^{\mu,x}$ solves \eqref{E2} with initial value $x\in\R^d$ and $P_{s,t}^\mu$ is defined in \eqref{Ptf}.
By \eqref{Mon} and $\sigma_t(x,\gamma)=\sigma_t(\gamma)$, we have
\begin{align}\label{decmo}
2\<b_t(x,\mu_t)-b_t(y,\mu_t),x-y\>\leq K_1|x-y|^2,\ \ t\geq 0, x,y\in\R^d.
\end{align}
It is easy to derive from \eqref{decmo} and $\sigma_t(x,\gamma)=\sigma_t(\gamma)$ that
$$|X_{s,t}^{\mu,x}-X_{s,t}^{\mu,y}|\leq \e^{\frac{1}{2}K_1(t-s)}|x-y|,\ \ x,y\in\R^d,0\leq s\leq t,$$
which implies that for any  $x\in\R^d,0\leq s\leq t, f\in C_b^\infty(\R^d)$,
\begin{align}\label{cdy} |\nabla P_{s,t}^\mu f|(x)&:=\limsup_{y\to x}\frac{|P_{s,t}^\mu f(y)-P_{s,t}^\mu f(x)|}{|y-x|} \leq \e^{\frac{1}{2}K_1(t-s)}P_{s,t}^\mu |\nabla f|(x).
\end{align}
Next, we divide into three steps to complete the proof.

(i) We first assume {\bf(A)}, \eqref{ups} and that for any $i\geq 1$,  there exists a locally bounded function $\ell^{i}:[0,\infty)\to[0,\infty)$ such that for any  $\mu_\cdot\in C([0,\infty),\scr P_2(\R^d))$,
\begin{align}\label{nabti}
\|\nabla^ib_t(\cdot,\mu_t)\|_\infty\leq \ell^{i}_{t}, \ \ t\geq 0, i\geq 1.
\end{align}
\eqref{nabti} and $\sigma_t(x,\gamma)=\sigma_t(\gamma)$ imply that $P_{r,t}^\mu  C_b^\infty(\R^d)\subset C_b^\infty(\R^d)$.
By It\^{o}'s formula, we conclude that
\begin{align}\label{kty}
P_{r,t}^\mu f-P_{r,s}^\mu f=\int_s^t P_{r,\theta}^\mu \scr L_\theta^\mu f\d \theta,\ \ f\in C_b^\infty(\R^d),t\geq s\geq r,
\end{align}
which together with the semigroup property implies that for $\varepsilon\geq 0$ and any $f\in C_b^\infty(\R^d)$,
\begin{align}\label{mty}
P_{r+\vv,t}^\mu f-P_{r,t}^\mu f&=P_{r+\vv,t}^\mu f-P_{r,r+\vv}^\mu P_{r+\vv,t}^\mu f=-\int_r^{r+\vv}P_{r,\theta}^\mu\scr L_\theta^\mu P_{r+\vv,t}^\mu f\d \theta.
\end{align}
By \eqref{kty}, \eqref{mty}, \eqref{nabti} and  $\sup_{t\in[0,T]}\|\sigma_t(\mu_t)\|<\infty$ for any $T>0$, we may use the dominated convergence theorem to derive that for any $f\in C_b^\infty(\R^d)$,
\begin{align*}
&-\int_r^{r+\vv}P_{r,\theta}^\mu\scr L_\theta^\mu P_{r+\vv,t}^\mu f\d \theta\\
&=-\int_r^{r+\vv}\scr L_\theta^\mu P_{\theta,t}^\mu f\d \theta+\int_r^{r+\vv}(P_{r,r}^\mu\scr L_\theta^\mu P_{\theta,t}^\mu f-P_{r,\theta}^\mu\scr L_\theta^\mu P_{r+\vv,t}^\mu f)\d \theta\\
&=-\int_r^{r+\vv}\scr L_\theta^\mu P_{\theta,t}^\mu f\d \theta +o(\vv),\ \ f\in C_b^\infty(\R^d).
\end{align*}
This together with \eqref{mty} implies the Kolmogorov backward equation, i.e.
\begin{align}\label{backw}
\frac{\d P_{r,t}^\mu f}{\d r}=-\scr L_r^\mu P_{r,t}^\mu f,\ \ f\in C_b^\infty(\R^d),\ \ a.s. \ \ r\in[s,t].
\end{align}
Since $\scr L_t^\mu$ given in \eqref{ger} is a diffusion operator (see e.g. \cite[Definition 1.11.1]{BGL}),  we have
\begin{align}\label{Lrf}\scr L_r^\mu \big(g\log g\big)
&=\frac{1}{2}g^{-1}\Gamma_1^{r,\mu}(g)+(\log g+1)\scr L_r^\mu g,\ \ g\in C_b^\infty(\R^d), g>0.
\end{align}
By It\^{o}'s formula, \eqref{backw} and \eqref{Lrf}, we derive that for $f\in C_b^\infty(\R^d)$, $f>0$,
\begin{align*}&((P_{s',t}^\mu f)\log(P_{s',t}^\mu f))(X_{s,s'}^{\mu,x})-P_{s,t}^\mu f(x)\log P_{s,t}^\mu f(x)\\
&=\int_s^{s'}\left(-(\log (P_{r,t}^\mu f)+1)\scr L_r^\mu P_{r,t}^\mu f+\scr L_r^\mu ((P_{r,t}^\mu f)\log(P_{r,t}^\mu f))\right)(X_{s,r}^{\mu,x})\d r+ M_{s'}\\
&=\int_s^{s'}\frac{1}{2}\Big((P_{r,t}^\mu f)^{-1}\big(\Gamma_1^{r,\mu}(P_{r,t}^\mu f)\big)\Big)(X_{s,r}^{\mu,x})\d r+M_{s'}, \ \ s'\in[s,t]
\end{align*}
for some martingale $\{M_{s'}\}_{s'\in[s,t]}$.
Combining this for $s'=t$ and taking expectation, we deduce from \eqref{ups} and \eqref{cdy} that
\begin{equation*}
 \begin{split}&P_{s,t}^\mu(f\log f)(x)-P_{s,t}^\mu f(x)\log P_{s,t}^\mu f(x)\\
 & = \frac{1}{2}\int_s^t\E\big((P_{r,t}^\mu f)^{-1}(\Gamma_1^{r,\mu}(P_{r,t}^\mu f))\big)(X_{s,r}^{\mu,x})\d r\\
&\le\frac{1}{2}\delta_1\int_s^t\E\big((P_{r,t}^\mu f)^{-1}|\nabla P_{r,t}^\mu f|^2\big)(X_{s,r}^{\mu,x})\d r\\
&\le \frac{1}{2}\delta_1\int_s^t\E\big((P_{r,t}^\mu f)^{-1}\e^{K_1(t-r)}(P_{r,t}^\mu |\nabla f|)^2\big)(X_{s,r}^{\mu,x})\d r,\ \ f\in C_b^\infty(\R^d),\ \ f>0.
\end{split}
\end{equation*}
This, combining with the following fact that
\begin{align}\label{csi}
(P_{r,t}^\mu |\nabla f|)^2=4\big(P_{r,t}^\mu \big( |\nabla f^{\frac{1}{2}}| f^{\frac{1}{2}}\big)\big)^2\leq 4 P_{r,t}^\mu  |\nabla f^{\frac{1}{2}}|^2 P_{r,t}^\mu f,
\end{align}
and the semigroup property $P_{s,r}^\mu P_{r,t}^\mu=P_{s,t}^\mu$,  yields \eqref{til}.

Simply denote $P_{t}^{\mu}:={P}_{0,t}^{\mu}$.
Integrating with respect to $\gamma$ in \eqref{til} with $s=0$,
we derive for $0<f\in C_b^\infty(\R^d)$,
\begin{align}\label{0mi} \gamma(P_t^\mu(f\log f))-\gamma(P_{t}^{\mu}f\log P_{t}^{\mu}f) \leq 2\delta_1\frac{\e^{K_1 t}-1}{K_1}\gamma\left(P_t^\mu\left(|\nabla   f^{\frac{1}{2}}|^2\right)\right).
\end{align}
By \eqref{ilo}, \eqref{cdy}, $P_{r,t}^\mu  C_b^\infty(\R^d)\subset C_b^\infty(\R^d)$ and \eqref{csi} for $r=0$, we conclude  for $0<f\in C_b^\infty(\R^d)$,
\begin{align*}\gamma(P_{t}^{\mu}f\log P_{t}^{\mu}f)&\leq \gamma(P_{t}^{\mu}f)\log \gamma(P_{t}^{\mu}f)+2L_0\gamma\big( |\nabla (P_{t}^{\mu}f)^{\frac{1}{2}}|^2 \big)\\
&\leq \gamma(P_{t}^{\mu}f)\log \gamma(P_{t}^{\mu}f)+2L_0\e^{K_1t}\gamma\left(P_t^\mu\left(|\nabla   f^{\frac{1}{2}}|^2\right)\right).
\end{align*}
This, together with \eqref{0mi}, implies
\begin{equation}\label{rlo13}
\begin{split}
&\gamma(P_t^\mu(f\log f))-\gamma(P_{t}^{\mu}f)\log \gamma(P_{t}^{\mu}f)\\
 &\leq 2\Big(L_0\e^{K_1t}+ \frac{\delta_1}{K_1}(\e^{K_1t}-1)\Big) \gamma\left(P_t^\mu\left(|\nabla   f^{\frac{1}{2}}|^2\right)\right),\ \ 0<f\in C_b^\infty(\R^d),t\geq 0.
\end{split}
\end{equation}

(ii) Assume {\bf(A)}, \eqref{ups} and that there exists a locally bounded function $L:[0,\infty)\to[0,\infty)$ such that for any $\mu_\cdot\in C([0,\infty),\scr P_2(\R^d))$,
\begin{align*}
|b_t(x,\mu_t)-b_t(\bar{x},\mu_t)|\leq L_t|x-\bar{x}|,\ \ t\geq 0,x,\bar{x}\in \R^d.
\end{align*}
Let $\hat{b}_t^\mu(x)=b_t(x,\mu_t)$ and $\hat{\sigma}_t^\mu=\sigma_t(\mu_t)$ and for any $m\geq 1$, let $\hat{b}^{m,\mu}$ be defined in the same way as $\hat{b}^m$ in \eqref{bmt} by replacing $\hat{b}$ with $\hat{b}^\mu$. Let $P_{s,t}^{m,\mu}$ be the associated semigroup to
$$\d \hat{X}_{s,t}^{m,\mu}=\hat{b}^{m,\mu}_t(\hat{X}_{s,t}^{m,\mu})\d t+\hat{\sigma}_{t}^\mu\d W_t, \ \ t\geq s\geq 0.$$
For $\vv>0$ and $f\in C_b^\infty(\R^d)$ with $f>0$, let $f_\vv=f+\vv$. By \eqref{ktl13}, \eqref{decmo} and \eqref{bmm} for $\hat{b}^{m,\mu}$ replacing $\hat{b}^{m}$ and step (i), \eqref{til} and \eqref{rlo13} hold for $(P_{s,t}^{m,\mu}, f_\vv)$ replacing $(P_{s,t}^{\mu},f)$. Letting $m\to\infty$, we get \eqref{til} and \eqref{rlo13} for $f_\vv$ replacing $f$ by the fact $\inf_{x>0}(x\log x)=-\e^{-1}$, Fatou's lemma, \eqref{supmh} for $\hat{X}^{m,\mu}$ replacing $\hat{X}^{m}$, the dominated convergence theorem and Lemma \ref{aky} below for $(\hat{b}^\mu,\hat{b}^{m,\mu})$ replacing $(\hat{b},\hat{b}^{m})$. Letting $\vv\to0$, \eqref{til} and \eqref{rlo13} hold by  Fatou's lemma and the monotone convergence theorem.

(iii) Finally, we assume {\bf (A)} and \eqref{ups}. Let $\hat{b}_t^\mu(x)=b_t(x,\mu_t)$ and $\hat{\sigma}_t^\mu=\sigma_t(\mu_t)$ and for any $n\geq 1$, let $\hat{b}^{(n),\mu}$ be defined in the same way as $\hat{b}^{(n)}$ in \eqref{bhn} by replacing $\hat{b}$ with $\hat{b}^\mu$.
Let $P_{s,t}^{(n),\mu}$ be the associated semigroup to
$$\d \hat{X}_{s,t}^{(n),\mu}=\hat{b}^{(n),\mu}_t(\hat{X}_{s,t}^{(n),\mu})\d t+\hat{\sigma}_t^\mu\d W_t,\ \ t\geq s\geq 0.$$
By \eqref{bnl} and \eqref{jg4dv7} for $\hat{b}^{(n),\mu}$ replacing $\hat{b}^{(n)}$ and step (ii), \eqref{til} and \eqref{rlo13} hold for $(P_{s,t}^{(n),\mu},f_\vv)$ replacing $(P_{s,t}^{\mu},f)$. Repeating the same argument in step (ii), we obtain \eqref{til} and \eqref{rlo13} by Lemma \ref{yap} in appendix for $(\hat{b}^\mu,\hat{b}^{(n),\mu})$ replacing $(\hat{b},\hat{b}^{(n)})$.
Taking $\mu_t=P_t^\ast\gamma, t\geq 0$ in \eqref{rlo13} and noting that $$(P_t^\ast\gamma)(f)=\gamma(P_t^\mu(f)), \ \ f\geq 0$$ and $\inf_{x>0}(x\log x)=-\e^{-1}$,
we derive \eqref{rlo}. The proof is completed.
\end{proof}
\vspace{-1em}

\begin{rem}
In \cite{CT}, the authors derive functional inequalities for the heat flow on Riemannian manifolds with time-dependent metrics. More precisely, they consider the operator $L_t=\Delta_t+Z_t$ for some vector field $Z_t$ on a complete Remannian manifold $M$ with time--dependent Riemannian metric $g_t$. They use the following curvature condition:
\begin{align}\label{CGC}\nonumber\mathfrak{R}_t^Z(X,X)&:=\mathrm{Ric}_t(X,X)-\<\nabla_X^t Z_t,X\>-\frac{1}{2}\partial_t g_t(X,X)\geq k(t),\\
	&\ \  X\in T_{x}M \text{ such that }g_t(X,X)=1, x\in M, k\in C([0,\infty)\times M).
\end{align}

When $\sigma$ depends on the spatial variable,  \eqref{cdy} does not hold.  Instead, if in addition $\sigma_t^\mu(\sigma_t^\mu)^\ast$ is invertible, one may  define a Riemannian metric as
\begin{align*}
g_t(X,Y)=\<(\sigma_t^\mu(\sigma_t^\mu)^\ast)^{-1} X,Y\>.
\end{align*}
Then the functional inequalities for $P_{s,t}^{\mu}$ follow by using the conclusions in \cite{CT} provided that \eqref{CGC} holds. However, as explained in \cite[p.1451]{W2011}, the assumption \eqref{CGC} involves the term $\nabla^2\big((\sigma_t^\mu(\sigma_t^\mu)^\ast)^{-1}\big)$, which is more or less implicit.  This is the reason why we
only consider the case $\sigma_t(x,\gamma)=\sigma_t(\gamma)$.
\end{rem}



\subsection{Time-homogeneous decoupled SDEs}

In the following, we consider the time-homogeneous version of \eqref{E1} given by
	\begin{equation}\label{EKy}
		\d X_t=b(X_t,\L_{X_t})\d t+\sigma(X_t,\L_{X_t})\d W_t.
	\end{equation}

 We make use of the following partially dissipative condition.
\begin{enumerate}
\item[{\bf(H)}] $b\colon \R^d\times \scr P_2(\R^d)\to\R^d$ is bounded on bounded sets. For any $\mu\in\scr P_2(\R^d)$, $b(\cdot,\mu)$ and $\sigma(\cdot,\mu)$ are continuous in $\R^d$. There exist constants $K_1>0, K_2>0, K_I>0, r_0>0$ and $\delta_1\geq\delta_2>0$ such that for any $x,y\in\R^d, \gamma,\tilde{\gamma}\in\scr P_2(\R^d)$,
\begin{align}\label{Mon1}\nonumber&2\<b(x,\gamma)-b(y,\tilde{\gamma}),x-y\>+\|\sigma(x,\gamma)-\sigma(y,\tilde{\gamma})\|_{HS}^2\\
&\leq K_1|x-y|^21_{\{|x-y|\leq r_0\}}-K_2|x-y|^21_{\{|x-y|> r_0\}}+K_I\W_2(\gamma,\tilde{\gamma})^2,
\end{align}
and
\begin{align}\label{sli}\delta_2I_{d\times d}\leq \sigma\sigma^\ast\leq  \delta_1 I_{d\times d}.
\end{align}

\end{enumerate}

In this paper, we aim to prove
exponential ergodicity of  the time-homogeneous version \eqref{EKy} in $\W_2$ and relative entropy under the partially dissipative condition {\bf(H)}.

To this end, we first prove a log-Sobolev inequality for the invariant probability measure $\Phi(\mu)$ of the time-homogeneous decoupled SDEs, i.e. for any $\mu\in\scr P_2(\R^d)$ we consider
	\begin{align}\label{ace}\d \tilde{X}_t^\mu=b(\tilde{X}_t^\mu,\mu)\d t+\sigma(\tilde{X}_t^\mu,\mu)\d W_t.
	\end{align}
We should remark that the log-Sobolev inequality for $\Phi(\mu)$ is uniform in $\mu$ with bounded second moments.
	We denote by $(\tilde{P}_t^\mu)^\ast\nu$ the distribution of the solution to \eqref{ace} with initial distribution $\nu\in\scr P_2(\R^d)$ and
	associated time-homogeneous semigroup
	$$\tilde{P}_t^\mu f(x)=\int_{\R^d}f\d \big((\tilde{P}_t^\mu)^\ast\delta_x\big),\ \ f\in\scr B_b(\R^d),x\in\R^d,t\geq 0.$$

\subsubsection{Invariant probability measure}
\begin{thm}\label{thm: IPM}
Assume {\bf (H)}. Then the following assertions hold.
\begin{enumerate}
\item[$(1)$] For any $\mu\in\scr P_2(\R^{d})$, \eqref{ace} has a unique invariant probability measure (IPM) denoted by $\Phi(\mu)$. Moreover, if $K_I<K_2$, there exists an $M>0$ such that
$\Phi$ is a map on $\hat{\scr P}_{2,M}(\R^{d})$, where
$$\hat{\scr P}_{2,M}(\R^{d}):=\left\{\mu\in\scr P_2(\R^{d}): \mu(|\cdot|^2)\leq M 
\right\}.$$
\item[$(2)$]
There exist constants $\vv>0$ and $C_1>0$ such that
\begin{align}\label{emo}
(\Phi(\mu))(\e^{\vv|\cdot|^2})\le C_1, \ \ \mu\in\hat{\scr P}_{2,M}(\R^d).
\end{align}
\end{enumerate}
\end{thm}
\vspace{-2em}
\begin{rem} We give some comments on the existence and uniqueness for invariant probability measure in Theorem \ref{thm: IPM}(1). Under {\bf(H)}, when  $\sigma=\sigma(\mu)$ and $b$ is locally Lipschitz continuous in spatial variable, one may use reflection couplings as in \cite[Corollary 3, Lemma 1]{Eberle} to derive the exponential ergodicty in $L^1$-Wasserstein distance,
see also \cite[Theorem 13]{CG} for the multiplicative noise case where $\sigma$ and $b$ are assumed to be locally Lipschitz continuous in spatial variable. We should remark that in Theorem \ref{thm: IPM}(1) above, the locally Lipschitz continuity of $b$ and $\sigma$ in spatial variable is not required.
Another alternative method to derive the existence and uniqueness for invariant probability measure is to apply \cite[Theorem 5.2(c)]{DMT}, which claims that for  a continuous time Markov process with the extended generator $\tilde{L}$, if there exists $V\geq 1$ such that for some petite set $B$, some constants $c>0,b_0>0$,
$$\tilde{L}V\leq -c V+b_01_{B},$$
then it is $V$-uniformly ergodic. Observe that the existence of petite set is the critical point in \cite[Theorem 5.2(c)]{DMT} and the level set $\{x:V(x)\leq n\}$ for some $n\geq 1$ may be the candidate.
\end{rem}
\begin{proof}
(1) Taking $y={\bf0}$ in \eqref{Mon1}, we get \begin{align}\label{Mon1t}\nonumber&2\<b(x,\mu)-b({\bf0},\delta_{\bf0}),x\>+\|\sigma(x,\mu)-\sigma({\bf0},\delta_{\bf0})\|_{HS}^2\\
&\leq K_1|x|^21_{\{|x|\leq r_0\}}-K_2|x|^21_{\{|x|> r_0\}}+K_I\|\mu\|_2^2,\ \ x\in\R^d,\mu\in\scr P_2(\R^d)
\end{align}
So, it follows from \eqref{sli} and \eqref{Mon1t} that
\begin{equation}\label{Ren1}
\begin{split}
&2\<x,b(x,\mu)\>+\|\sigma(x,\mu)\|_{\rm HS}^2\\
&\le -K_2|x|^2+(K_1+K_2)r_0^2+2|x||b({\bf0},\delta_{\bf0})|+K_I\mu(|\cdot|^2)\\
&\quad+\| \sigma({\bf0},\delta_{\bf0})\|_{\rm HS}^2+2\|\sigma(x,\mu)-\sigma({\bf0},\delta_{\bf0})\|_{HS} \|\sigma({\bf0},\delta_{\bf0})\|_{\rm HS}\\
&\le -K_2|x|^2+2|x||b({\bf0},\delta_{\bf0})|+K_I\mu(|\cdot|^2)+C_0,\ \ x\in\R^d,\mu\in\scr P_2(\R^d)
\end{split}
\end{equation}
for some constant $C_0>0$.
Whence, an application of \cite[Lemma 2.7(1)]{ZSQ1}, along with \eqref{sli} and \eqref{Ren1}, yields that \eqref{ace} has an IPM.

By \cite[Theorem 1.1(2)]{W2011} and {\bf(H)}, one has
Wang's Harnack inequality, i.e. there exists $p_0>1$ such that for any $p>p_0$, one can find a constant $C>0$ such that
\begin{align}\label{WHa}(\tilde{P}_t^\mu|f|)^p\leq \tilde{P}_t^\mu|f|^p\exp\left(C\frac{|x-y|^2}{1-\e^{-K_1 t}}\right),\ \ f\in\scr B_b(\R^d), x,y\in\R^d, t>0.
\end{align}
So, \cite[Theorem 1.4.1(3)]{Wbook} implies that \eqref{ace} has a unique IPM written as $\Phi(\mu).$

Next, once $K_I<K_2$ is valid, \cite[Lemma 2.8]{ZSQ1} implies that, for some constant $M>0,$  $\hat{\scr P}_{2,M}(\R^{d})$ is an invariant set of $\Phi$ so the second assertion in (1) holds true.

(2) It\^o's formula, \eqref{sli} as well as  \eqref{Ren1} imply  that for $\vv=\frac{K_2}{4\delta_1}$ and $\mu\in\hat{\scr P}_{2,M}(\R^{d}),$
\begin{align}\label{gpy}
 \d \e^{\vv|\tilde{X}_t^\mu|^2}
\nonumber&\leq \vv\e^{\vv|\tilde{X}_t^\mu|^2}\left( -K_2|\tilde{X}_t^\mu|^2+2|\tilde{X}_t^\mu||b({\bf0},\delta_{\bf0})|+K_I\mu(|\cdot|^2)+C_0+2\vv\delta_1 |\tilde{X}_t^\mu|^2\right)\d t\\
&\quad+2\vv\e^{\vv|\tilde{X}_t^\mu|^2}\<\sigma(\tilde{X}_t^\mu,\mu)\d W_t, \tilde{X}_t^\mu\>\\
\nonumber\le &\vv\e^{\vv|\tilde{X}_t^\mu|^2}\left( -\frac{K_2}{4}|\tilde{X}_t^\mu|^2+\frac{4|b({\bf0},\delta_{\bf0})|^2}{K_2}+K_IM+C_0\right)\d t\\
\nonumber&+2\vv\e^{\vv|\tilde{X}_t^\mu|^2}\<\sigma(\tilde{X}_t^\mu,\mu)\d W_t, \tilde{X}_t^\mu\>.
\end{align}
Note that we can find constants $c_1(\vv), c_2(\vv)>0$ such that
\begin{align*} &\vv\e^{\vv|\tilde{X}_t^\mu|^2} \left(-\frac{K_2}{4}|\tilde{X}_t^\mu|^2+\frac{4|b({\bf0},\delta_{\bf0})|^2}{K_2}+K_IM+C_0\right)\leq c_1(\vv)-c_2(\vv)\e^{\vv|\tilde{X}_t^\mu|^2}.
\end{align*}
Substituting this into \eqref{gpy} and taking $\tilde{X}_0^\mu={\bf0}$ enables us to derive that for some constant $C_1>0$
\begin{align*}
 \int_0^t\E(\e^{\vv|\tilde{X}_s^\mu|^2})\,\d s\le C_1t.
\end{align*}
This obviously implies \eqref{emo} by a tightness argument.\end{proof}
\vspace{-1em}
\begin{rem} Our proof of the IPM's uniqueness improves the approach in \cite[Lemma 2.7(2)]{ZSQ1} by using Wang's Harnack inequality \cite[Theorem 1.1(2)]{W2011}. This allows us to eliminate the need for  the local Lipschitz of $b(\cdot,\mu), \sigma(\cdot,\mu) $ and the linear growth of $b(\cdot,\mu)$.
\end{rem}
\subsubsection{Wang's Harnack inequality and hyperboundedness}

To derive the hyperboundedness of $\tilde{P}_t^\mu$, a suitable Wang's Harnack inequality is needed. In \cite[Lemma 2.6]{Wang2020b}, the author established a Harnack inequality for \eqref{ace} in the case where the drift $b$ and diffusion coefficient $\sigma$ do not depend on the measure variable $\mu$. In the following theorem, we derive a uniform Wang's Harnack inequality for \eqref{ace} with respect to the measure parameter $\mu$. The proof is completely the same with that of \cite[Lemma 2.6]{Wang2020b}. Since \cite{Wang2020b} is written in Chinese, for reader's convenience, we will provide the proof.
Note that \eqref{niHar} is different form \eqref{WHa}, the factor $\e^{K_2t}-1$ in \eqref{niHar} is crucial for us to derive the hyperboundedness of $\tilde{P}_t^\mu$ below.

\begin{lem} \label{GTY} Assume {\bf (H)}. Then the following assertions hold.

(1) There exist constants $p_0>1$ depending on $\delta_1,\delta_2$, $h_1>0$ depending on $K_1, r_0$, $\delta_1,K_2$ and $h_2>0$ depending on $\delta_1,K_2$ such that Wang's Harnack inequality holds, i.e. for any $p\geq p_0$, $t>0$,$\mu\in\scr P_2(\R^d), x,y\in\R^d, f\in\scr B_b(\R^d)$,
\begin{align}\label{niHar}
(\tilde{P}_t^\mu |f|)^p(x)&\leq (\tilde{P}_t^\mu |f|^p)(y)\exp\left((p-1)\left(h_1 t+\frac{h_2K_2|x-y|^2}{\e^{K_2t}-1}\right)\right) .
\end{align}
(2) There exist constants $t_0>0, p>1, C_2>0$ such that
\begin{align*}
\|\tilde{P}_{t_0}^\mu\|_{L^p(\Phi(\mu))\to L^{2p}(\Phi(\mu))}\leq C_2, \ \ \mu\in\hat{\scr P}_{2,M}(\R^d).
\end{align*}
\end{lem}
\begin{proof}

Let $\xi_t=\frac{1}{K_2}(\e^{-K_2(t-t_0)}-1), t\in[0,t_0], $  for fixed $t_0>0$. It is easy to see that
\begin{align}
\label{xit}\xi_t'+K_2\xi_t=-1, \ \ t\in[0,t_0].
\end{align}
Consider the following coupling:
\begin{equation}\label{Ren}
\begin{cases}
\d \tilde{X}_t^\mu=b(\tilde{X}_t^\mu,\mu)\d t+\sigma(\tilde{X}_t^\mu,\mu)\d W_t,\ \ \tilde{X}_0^\mu=x,\\
\d Y_t= b(Y_t,\mu)\d t+\sigma(Y_t,\mu)\Big(\d W_t+ \hat\sigma(\tilde{X}_t^\mu,\mu)\frac{\tilde{X}_t^\mu-Y_t }{\xi_t}\d t\Big),\ \ Y_0=y, t\in[0,t_0),
\end{cases}
\end{equation}
where  $\hat \si(x,\mu):=[\sigma^\ast(\sigma\sigma^\ast)^{-1}](x,\mu)$.
Set
\begin{align*}
\d \tilde{W}_t:=\d W_t+\hat\sigma(\tilde{X}_t^\mu,\mu)\frac{\tilde{X}_t^\mu-Y_t}{\xi_t}\d t.
\end{align*}
Obviously, \eqref{sli} yields that
\begin{align}\label{etasq}
\left|\hat\sigma(\tilde{X}_t^\mu,\mu)\frac{\tilde{X}_t^\mu-Y_t}{\xi_t}\right|^2\leq \delta_2^{-1}\frac{|\tilde{X}_t^\mu-Y_t|^2}{\xi_t^2}.
\end{align}
According to {\bf(H)}, as explained in the proof of \cite[Theorem 1.1]{W2011}, $(\tilde{W}_t)_{t\in[0,t_0]}$ is an $n$-dimensional Brownian motion under $\d\Q=R_{t_0}\d\P$, where
\begin{align*}
R_{s}=\exp\left(\int_0^{s}-\left\<\hat\sigma(\tilde{X}_t^\mu,\mu)\frac{\tilde{X}_t^\mu-Y_t}{\xi_t},\d W_t\right\>-\frac{1}{2}\int_0^{s}\left|\hat\sigma(\tilde{X}_t^\mu,\mu)\frac{\tilde{X}_t^\mu-Y_t}{\xi_t}\right|^2\d t\right)
\end{align*}
with $\{R_s\}_{s\in[0,t_0]}$ being a uniformly integrable martingale and $\Q$-a.s. $\tilde{X}_{t_0}^\mu=Y_{t_0}$.
Notice that \eqref{Ren} can be reformulated as follows:
\begin{equation*}
\begin{cases}
\d \tilde{X}_t^\mu =b(\tilde{X}_t^\mu,\mu)\d t+\sigma(\tilde{X}_t^\mu,\mu)\d \tilde{W}_t-\frac{\tilde{X}_t^\mu-Y_t}{\xi_t}\d t,\ \ \tilde{X}_0^\mu=x,\\
\d Y_t=b(Y_t,\mu)\d t+\sigma(Y_t,\mu)\d \tilde{W}_t,\ \ Y_0=y.
\end{cases}
\end{equation*}
Observe that {\bf(H)} implies
\begin{align*}
&2\<b(x,\mu)-b(y,\mu),x-y\>+\|\sigma(x,\mu)-\sigma(y,\mu)\|_{HS}^2\\
&\leq (K_1+K_2)|x-y|^21_{\{|x-y|\leq r_0\}}-K_2|x-y|^2\\
&\leq (K_1+K_2)r_0|x-y|-K_2|x-y|^2.
\end{align*}
By It\^{o}'s formula, we derive from this and  \eqref{xit} that
\begin{align}\label{zxi}
\nonumber\d \frac{|\tilde{X}_t^\mu-Y_t|^2}{\xi_t}&=-\frac{(\xi_t'+2)|\tilde{X}_t^\mu-Y_t|^2}{\xi_t^2}\d t +\frac{2}{\xi_t}\<(\sigma(\tilde{X}_t^\mu,\mu)-\sigma(Y_t,\mu))\d \tilde{W}_t, \tilde{X}_t^\mu-Y_t\>\\
\nonumber&\quad+\frac{1}{\xi_t}\big( 2\<b(\tilde{X}_t^\mu,\mu)-b(Y_t,\mu),\tilde{X}_t^\mu-Y_t\> +
 \|\sigma(\tilde{X}_t^\mu,\mu)-\sigma(Y_t,\mu)\|_{HS}^2\big)\d t\\
\nonumber&\leq -\frac{|\tilde{X}_t^\mu-Y_t|^2}{\xi_t^2}\d t+\frac{(K_1+K_2)r_0|\tilde{X}_t^\mu-Y_t|}{\xi_t}\d t\\
&\qquad\quad+\frac{2\<(\sigma(\tilde{X}_t^\mu,\mu)-\sigma(Y_t,\mu))\d \tilde{W}_t, \tilde{X}_t^\mu-Y_t\>}{\xi_t}\\
\nonumber&\leq -\frac{1}{2}\frac{|\tilde{X}_t^\mu-Y_t|^2}{\xi_t^2}\d t+\frac{1}{2}(K_1+K_2)^2r_0^2\d t\\
\nonumber&\qquad\quad+\frac{2\<(\sigma(\tilde{X}_t^\mu,\mu)-\sigma(Y_t,\mu))\d \tilde{W}_t, \tilde{X}_t^\mu-Y_t\>}{\xi_t},\ \ t<t_0.
\end{align}
For any $r>0$, it follows from \eqref{zxi} that
\begin{align*}
&\E_{\Q}\exp\left(r\int_0^{t_0}\frac{|\tilde{X}_t^\mu-Y_t|^2}{\xi_t^2}\d t\right)\\
&\leq \exp\left(2r\frac{|x-y|^2}{\xi_0}+r(K_1+K_2)^2r_0^2t_0\right)\\
&\qquad\quad\times\E_{\Q}\exp\left(r\int_0^{t_0}\frac{4\<(\sigma(\tilde{X}_t^\mu,\mu)-\sigma(Y_t,\mu))\d \tilde{W}_t, \tilde{X}_t^\mu-Y_t\>}{\xi_t}\right)\\
&\leq\exp\left(2r\frac{|x-y|^2}{\xi_0}+r(K_1+K_2)^2r_0^2t_0\right)\\
&\qquad\quad\times\left(\E_{\Q}\exp\left(128r^2\delta_1\int_0^{t_0}\frac{|\tilde{X}_t^\mu-Y_t|^2}{\xi_t^2}\right)\right)^{\frac{1}{2}}.
\end{align*}
where in the last step, we used $\|\sigma(x_1,\mu)^\ast-\sigma(x_2,\mu)^\ast\|^2\leq 4\delta_1$ and that  for a continuous exponential integrable martingale $M_t$, it holds
$$\E\e^{M_t}\leq (\E\e^{2\<M\>_t})^{\frac{1}{2}}.$$
Taking $r=\frac{1}{128\delta_1}$ and using $\xi_0=\frac{1}{K_2}(\e^{K_2t_0}-1)$, we arrive at
\begin{align}\label{zetsq}
\nonumber &\E_{\Q}\exp\left(\frac{1}{128\delta_1}\int_0^{t_0}\frac{|\tilde{X}_t^\mu-Y_t|^2}{\xi_t^2}\d t\right)\\
&\leq \exp\left(\frac{1}{32\delta_1}\frac{K_2|x-y|^2}{\e^{K_2t_0}-1}+\frac{1}{64\delta_1}(K_1+K_2)^2r_0^2t_0\right).
\end{align}
Let $\eta_t=-\hat\sigma(\tilde{X}_t^\mu,\mu)\frac{\tilde{X}_t^\mu-Y_t}{\xi_t}$. Observe that for any $p>1$, it holds
\begin{align*}
\E R_{t_0}^{\frac{p}{p-1}}=\E_{\Q}R_{t_0}^{\frac{1}{p-1}}
&= \E_{\Q}\Bigg(\exp\left((p-1)^{-1}\int_0^{t_0}\<\eta_t,\d \tilde{W}_t\>-(p-1)^{-2}\int_0^{t_0}|\eta_t|^2\d t\right)\\
&\qquad\quad\times\exp\left(\left(\frac{1}{2}(p-1)^{-1}+(p-1)^{-2}\right)\int_0^{t_0}|\eta_t|^2\d t\right)\Bigg)\\
&\leq \left(\E_{\Q}\exp\left(\left((p-1)^{-1}+2(p-1)^{-2}\right)\int_0^{t_0}|\eta_t|^2\d t\right)\right)^{\frac{1}{2}}.
\end{align*}
Take $(p_0-1)^{-1}=\frac{1}{2}\wedge \frac{1}{256\delta_2^{-1}\delta_1}$. Then for any $p\geq p_0$, it holds
$$(p-1)^{-1}+2(p-1)^{-2}\leq (p_0-1)^{-1}+2(p_0-1)^{-2}<2(p_0-1)^{-1}\leq\frac{1}{128\delta_2^{-1}\delta_1}.$$ So, for any $p\geq p_0$, we derive from \eqref{etasq} and \eqref{zetsq} that
\begin{align}\label{Rtp}
\E R_{t_0}^{\frac{p}{p-1}}\leq \exp\left(\frac{1}{64\delta_1}\frac{K_2|x-y|^2}{\e^{K_2t_0}-1}+\frac{1}{128\delta_1}(K_1+K_2)^2r_0^2t_0\right).
\end{align}
Furthermore, it follows from the fact $\Q$-a.s. $\tilde{X}_{t_0}^\mu=Y_{t_0}$ and the H\"{o}lder inequality that for any $p>1$,
$$(\tilde{P}_{t_0}^\mu |f|(y))^p=(\E_{\Q}(|f|(Y_{t_0})))^p=(\E_{\Q}(|f|(\tilde{X}_t^\mu)))^p\leq \tilde{P}_{t_0}^\mu |f|^p(x)\left(\E R_{t_0}^{\frac{p}{p-1}}\right)^{p-1}.$$
Combining this with \eqref{Rtp}, we finish the proof.

(2) Observing (1) and following the strategy adopted in \cite[Corollary 1.3]{W2011}, it is more or less standard to derive from \eqref{niHar} that for some locally bounded function $(0,\infty)\ni t\mapsto C_t\in[0,\infty)$,
\begin{align*}
\sup_{\{\Phi(\mu)\}(|f|^p)\leq 1}(\tilde{P}_t^\mu |f|)^{2p}(x)
&\leq  C_t\exp\left( \frac{h_2K_2}{\e^{K_2t}-1}|x|^2\right),\ \ \mu\in\scr P_2(\R^d), x\in\R^d.
\end{align*}
Then, by taking $t_0>0$ such that $\frac{h_2K_2}{\e^{K_2t_0}-1}=\vv$, \eqref{emo} enables us to derive that
\begin{align}\label{hyper}
\|\tilde{P}_{t_0}^\mu\|_{L^p(\Phi(\mu))\to L^{2p}(\Phi(\mu))}^{2p}\leq C_{t_0}\big(\Phi(\mu)\big)(\e^{\vv|\cdot|^2})\leq C_{t_0}C_1, \ \ \mu\in\hat{\scr P}_{2,M}(\R^d).
\end{align}
The proof is completed.
\end{proof}
\subsubsection{Log-Sobolev inequality}


Before we prove the uniform log-Sobolev inequality for $\Phi(\mu)$ in $\mu$ with bounded second moments stated in Theorem \ref{logta0}, we give a general result Lemma \ref{RWlog}. The first assertion in Lemma \ref{RWlog} is known by \cite[Theorem 2.1(1)]{RW2003}, which states that the hyperboundedness together with log-Sobolev inequality of the semigroup implies the defective log-Sobolev inequality for the corresponding IPM. Let us remark that although \cite[Theorem 2.1(1)]{RW2003} requires a curvature lower bound, the proof of it reveals that only the log-Sobolev inequality of the semigroup is used.  Lemma \ref{RWlog}(2) is devoted to derive uniform log-Sobolev inequality from uniform defective log-Sobolev inequality for IPMs associated to a family of diffusion processes.  The tightness of IPMs and uniform local bounds on the densities of each IPM will play an important role.

\begin{lem}\label{RWlog}

\begin{enumerate}
\item[$(1)$]
Let $L=\frac12\mathrm{tr}(a\nabla^2)+\<b,\nabla\>$ be the generator of a conservative diffusion process on $\R^d$ with associated semigroup $(P_t)_{t\geq0}$.
Let $\bar{\mu}\in\mathscr P(\R^d)$ be the IPM of $(P_t)_{t\geq0}$. Assume that there exist  $t_0>0 $, $1<p<q<\infty$ and $C_0>0$ such that
$\|P_{t_0}\|_{p\to q }:= \|P_{t_0}\|_{L^p(\bar{\mu})\to L^q(\bar{\mu})}<C_0.$
If additionally the log-Sobolev inequality for $P_t$ holds, i.e.
\begin{align*}P_t(f\log f)-P_t f\log P_t f\leq 4h(t)P_t( |\nabla f^{\frac{1}{2}}|^2),\ \ t\geq 0, 0<f\in C_0^\infty(\R^d),
\end{align*}
for some measurable function $h:[0,\infty)\to[0,\infty)$, then
 $\bar{\mu}$ satisfies the following defective log-Sobolev inequality:
\begin{align*}&\bar{\mu}(f\log f)-\bar{\mu}(f)\log \bar{\mu}(f)\\
&\leq \frac{4p(q-1)}{q-p}h(t_0)\bar{\mu} ( |\nabla f^{\frac{1}{2}}|^2)+\frac{pq}{q-p}\log\|P_{t_0}\|_{p\to q },\ \ \bar{\mu}(f)=1, 0<f\in C_0^\infty(\R^d).
\end{align*}
\item[$(2)$] Let $L^\theta=\frac12\mathrm{tr}(a^\theta\nabla^2)+\<b^\theta,\nabla\>$, $\theta\in\Theta$ be a family of conservative diffusion processes on $\R^d$ with associated semigroup $(P_t^\theta)_{t\geq0}$ and IPM $\bar{\mu}^\theta$ satisfying the assumptions in (1) uniformly in $\theta$. Assume that $\{\bar{\mu}^\theta\}_{\theta\in\Theta}$ form a tight subset of $\mathscr P(\R^d)$ and that for any $R>0$, there exist positive constants $C_1(R)\leq C_2(R)$ such that $\bar{\rho}_\theta=\frac{\d \bar{\mu}^\theta}{\d x},\ \ \theta\in\Theta$ satisfy
\begin{align}\label{ulb}C_{1}(R)\leq \bar{\rho}_\theta(x)\leq C_2(R),\ \ x\in B_{R}:=\{x\in\R^d, |x|\leq R\}, \theta\in\Theta.
\end{align}
Then the weak Poincar\'{e} inequality holds, i.e.
$$\bar{\mu}^\theta(f^2)\leq
\alpha(r)\bar{\mu}^\theta ( |\nabla f|^2)+r\|f\|_\infty^2,\ \ r>0,0<f\in C_0^\infty(\R^d),\bar{\mu}^\theta(f)=0,\theta\in\Theta$$
for some decreasing function $\alpha:(0,\infty)\to[0,\infty).$
Moreover, the log-Sobolev inequality holds, that is there exists a constant $C_{LS}>0$ such that
\begin{align}\label{TLS}&\bar{\mu}^\theta(f\log f)-\bar{\mu}^\theta(f)\log \bar{\mu}^\theta(f)\leq C_{LS}\bar{\mu}^\theta (|\nabla f^{\frac{1}{2}}|^2),\ \  0<f\in C_0^\infty(\R^d),\theta\in\Theta.
\end{align}
\end{enumerate}
\end{lem}

\begin{proof} (1) follows from the proof of \cite[Theorem 2.1(1)]{RW2003} so that we omit it. It remains to prove (2). Observe that for any $R>0$, the classical Poincar\'{e} inequality for Lebesgue measure holds, i.e.
\begin{align}\label{LPI}\frac{\int_{B_R}f^2\d x}{|B_R|}\leq \frac{\left(\int_{B_R}f\d x\right)^2}{|B_R|^2}+ \frac{\beta(R)}{|B_R|}\int_{B_R} |\nabla f|^2\d x,\ \ f\in C_0^\infty(\R^d)
\end{align}
for some constant $\beta(R)>0$. When $(\bar{\mu}^\theta)_{\theta\in\Theta}$ is tight, for any $\vv\in(0,1)$, we can take $R>0$ large enough such that $\bar{\mu}^\theta(B_R)\geq 1-\vv, \theta\in\Theta$.
Let $\d \bar{\mu}^\theta_R=\frac{1_{B_R}}{\bar{\mu}^\theta(B_R)}\d \bar{\mu}^\theta$.
It follows from \eqref{ulb} and \eqref{LPI} that
\begin{align*}
&\int_{\R^d}(f-\bar{\mu}^\theta_R(f))^2\d \bar{\mu}^\theta_R\\
&=\inf_{c\in\R}\int_{\R^d}(f-c)^2\d \bar{\mu}^\theta_R\\
&=\inf_{c\in\R}\frac{|B_R|}{\bar{\mu}^\theta(B_R)}\int_{\R^d}(f-c)^2\frac{1_{B_R}}{|B_R|}\bar{\rho}_\theta\d x\\
&\leq\frac{\beta(R)C_2(R)}{\bar{\mu}^\theta(B_R)}\int_{\R^d} 1_{B_R}|\nabla f|^2\d x\\
&\leq\frac{\beta(R)C_2(R)}{C_1(R)\bar{\mu}^\theta(B_R)}\int_{\R^d} |\nabla f|^2\d \bar{\mu}^\theta,\ \ f\in C_0^\infty(\R^d),\theta\in\Theta.
\end{align*}
This means
\begin{align*}
&\int_{B_R}f^2\d \bar{\mu}^\theta\leq \frac{\left(\int_{B_R}f\d \bar{\mu}^\theta\right)^2}{\bar{\mu}^\theta(B_R)}+\frac{\beta(R)C_2(R)}{C_1(R)}\int_{\R^d} |\nabla f|^2\d \bar{\mu}^\theta,\ \ f\in C_0^\infty(\R^d),\theta\in\Theta.
\end{align*}
This implies that
$$\bar{\mu}^\theta(f^2)\leq
\frac{\beta(R)C_2(R)}{C_1(R)}\bar{\mu}^\theta ( |\nabla f|^2)+\frac{\vv}{1-\vv}\|f\|_\infty^2,\ \ f\in C_0^\infty(\R^d),\bar{\mu}^\theta(f)=0,\theta\in\Theta.$$
This together with the fact that $(\bar{\mu}^\theta)_{\theta\in\Theta}$ is tight implies the weak Poincar\'{e} inequality
$$\bar{\mu}^\theta(f^2)\leq \alpha(r)\bar{\mu}^\theta ( |\nabla f|^2)+r\|f\|_\infty^2,\ \ r>0, f\in C_0^\infty(\R^d),\bar{\mu}^\theta(f)=0,\theta\in\Theta
$$
holds for $$\alpha(r):=\inf\left\{\frac{\beta(R)C_2(R)}{C_1(R)}: \inf_{\theta\in\Theta}\bar{\mu}^\theta(B_R)\geq \frac{1}{1+r}\right\}.$$
Observe that due to \cite[Theorem 3.3.1]{W2005}, the uniform defective log-Sobolev inequality in $\theta$ due to (1) implies the defective Poincar\'{e} inequality
\begin{align}\label{dpi}\bar{\mu}^\theta(f^2)\leq C_1\bar{\mu}^\theta ( |\nabla f|^2)+C_2\bar{\mu}^\theta(|f|)^2,\ \ f\in C_0^\infty(\R^d),\theta\in \Theta.
\end{align}
We can take $r_0>0$ satisfying $c_0:=\frac{1}{2}(1+r_0+\sqrt{(C_2+1+r_0)r_0})<1$ such that
\begin{align*}\bar{\mu}^\theta(f^2)\leq \alpha(r_0)\bar{\mu}^\theta ( |\nabla f|^2)+r_0\|f\|_\infty^2,\ \ f\in C_0^\infty(\R^d),\bar{\mu}^\theta(f)=0,\theta\in\Theta.
\end{align*}
By \cite[Proposition 4.1.2(2)]{W2005}, we conclude from this and \eqref{dpi} that the Poincar\'{e} inequality holds:
$$\bar{\mu}^\theta(f^2)-\bar{\mu}^\theta(f)^2\leq \frac{C_1+\alpha(r_0)}{1-c_0}\bar{\mu}^\theta ( |\nabla f|^2),\ \ f\in C_0^\infty(\R^d), \theta\in \Theta.$$
This together with the defective log-Sobolev inequality implies \eqref{TLS}, see for instance the proof of \cite[Corollary 1.3]{Wang2014}.
\end{proof}

\begin{thm}\label{logta0}
Assume {\bf (H)} and  $\sigma(x,\gamma)=\sigma(\gamma)$. Then $\Phi(\mu)$ satisfies the log-Sobolev inequality
\begin{align}\label{losob0} \Phi(\mu)(f\log f)- \Phi(\mu)(f)\log\Phi(\mu)(f) \leq C_{L}\Phi(\mu)(|\nabla f^{\frac{1}{2}}|^2),\ \ \mu\in\hat{\scr P}_{2,M}(\R^{d})
\end{align}
for some constant $C_L>0$ and $ f\in C_0^\infty(\R^d)$ with $f>0$.
\end{thm}

\begin{rem} In the symmetric case on $\R^d$, for $\mu(\d x)=\e^{-V}\d x/\int_{\R^d}\e^{-V}\d x$, the Dirichlet form is $\scr E(f,f)=\mu(|\nabla f|^2)$. \cite[Theorem 3.17]{CGWW} uses the Lyapunov conditions to derive the super Poincar\'{e} inequality and hence defective log-Sobolev inequality. The main idea of the proof is first proving a local super Poincar\'{e} inequality and then use Lyapunov condition to obtain a super Poincar\'{e} inequality on the whole space. This method relies crucially on the explicit information of $V$. In Theorem \ref{logta0}, we only present some conditions on the coefficients while the invariant probability measure is implicit.
\end{rem}
\vspace{-1em}
\begin{proof}

Since $\sigma(x,\gamma)=\sigma(\gamma)$, by \eqref{til} with $s=0$, we have
\begin{align}\label{tillo} \nonumber&\tilde{P}_{t}^\mu(f\log f)-\tilde{P}_{t}^\mu(f)\log \tilde{P}_{t}^\mu(f)\\
 &\leq \frac{2\delta_1 }{K_1}(\e^{K_1 t}-1)\tilde{P}_{t}^\mu\big( |\nabla   f^{\frac{1}{2}}|^2 \big),\ \ t\geq 0, \mu\in\hat{\scr P}_{2,M}(\R^d), 0<f\in C_0^\infty(\R^d).
\end{align}
This, together with \eqref{hyper} and Lemma \ref{RWlog}(1), implies that the following  defective log-Sobolev inequality
\begin{align*}&\Phi(\mu)(f\log f)-\Phi(\mu)(f)\log \Phi(\mu)(f)\\
&\leq \frac{2\delta_1}{K_1}(\e^{K_1t_0}-1)(2p-1)\Phi(\mu) ( |\nabla f^{\frac{1}{2}}|^2 )+\log(C_{t_0}C_1),\ \ \mu\in\hat{\scr P}_{2,M}(\R^d), 0<f\in C_0^\infty(\R^d).
\end{align*}

Since $\{\Phi(\mu): \mu\in\hat{\scr P}_{2,M}(\R^{d})\}$ is tight due to \eqref{emo}, according to Lemma \ref{RWlog}(2), in order to derive the desired log-Sobolev inequality \eqref{losob0}, it remains to prove that  for any $R>0$, there exist positive constants $C_1$ and $C_2$ depending on $R$, $M$ such that
\begin{align}\label{upblo}C_1\leq \rho_\mu(x)\leq C_2,\ \ x\in B_R, \mu\in\hat{\scr P}_{2,M}(\R^{d})
\end{align}
for $\rho_\mu:=\frac{\d \Phi(\mu)}{\d x}$.
Since $b(x,\mu)$ is locally bounded in $(x,\mu)\in\R^d\times \scr P_2(\R^d)$, for any $R>0$, there exists $\eta(R,M)>0$ such that
$$ \sup_{|x|\leq R}|b(x,\mu)|\leq \eta(R,M),\ \ \mu\in\hat{\scr P}_{2,M}(\R^{d}).$$
 Noting $(\sigma\sigma^\ast)(\mu)\geq \delta_2I_{d\times d}$, by \cite[Corollary 1.2.11]{BKR}, for any $R>0$, there exists a constant $C_0$ depending on $R$, $\eta(R,M)$ and $\delta_2$ such that
$$\sup_{\bar{B}_R}\rho_\mu\leq C_0\inf_{\bar{B}_R}\rho_\mu,\ \ \mu\in\hat{\scr P}_{2,M}(\R^{d}).$$
Take $R>0$ such that $\inf_{\mu\in\hat{\scr P}_{2,M}(\R^{d})}\Phi(\mu)(\bar{B}_R)\geq \frac{1}{2}$, which can be ensured by \eqref{emo}. Then for any $x,y\in \bar{B}_R$, it holds
$$\rho_\mu(x)\leq C_0\rho_\mu(y),\ \ \mu\in\hat{\scr P}_{2,M}(\R^{d}).$$
This implies
$$\frac{1}{2}\leq\Phi(\mu)(\bar{B}_R)=\int_{\bar{B}_R}\rho_\mu(x)\d x\leq |\bar{B}_R|C_0\rho_\mu(y),\ \ y\in \bar{B}_R, \mu\in\hat{\scr P}_{2,M}(\R^{d}),$$
and
$$\rho_\mu(x)|\bar{B}_R|\leq \int_{\bar{B}_R}C_0\rho_\mu(y)\d y\leq C_0,\ \ x\in \bar{B}_R, \mu\in\hat{\scr P}_{2,M}(\R^{d}).$$
Therefore, \eqref{upblo} holds.
\end{proof}
\subsection{Exponential ergodicity of McKean-Vlasov SDEs}
We are now ready to prove exponential ergodicity of \eqref{EKy} in $L^2$-Wasserstein and relative entropy under the partially dissipative condition.
\beg{thm}\label{Erg} Assume {\bf (H)} and that $\sigma(x,\gamma)=\sigma(\gamma)$.
\begin{enumerate}
\item[$(1)$]
If $K_I$ is small enough, then  \eqref{EKy} has a unique IPM $\mu_\infty$ which satisfies the log-Sobolev inequality
\begin{align}\label{losob01}&\mu_\infty(f\log f)-\mu_\infty(f)\log \mu_\infty(f)\leq C_{L}\mu_\infty(|\nabla f^{\frac{1}{2}}|^2),\ \ 0<f\in C_0^\infty(\R^d)
\end{align}
for some constant $C_L>0$.
\item [$(2)$] Moreover, there exist constants $c>0, \lambda >0$ such that
\begin{align}\label{Wcove}\W_2(P_t^\ast\mu_0,\mu_\infty)\leq c\e^{-\lambda t}\W_2(\mu_0,\mu_\infty),\ \ t\geq 0,\mu_0\in\scr P_2(\R^d).
\end{align}
\item[$(3)$]
If in addition, $\sigma(x,\gamma)=\sigma$ and \begin{align}\label{stmon}&\nonumber\<b(x,\gamma)-b(y,\tilde{\gamma}),x-y\>\\
&\leq K|x-y|^2+K|x-y|\W_2(\gamma,\tilde{\gamma}),\ \ x,y\in\R^d, \gamma,\tilde{\gamma}\in\scr P_2(\R^d)
\end{align} for some constant $K\geq0$, then it holds
\begin{align}\label{Wcoveent}\mathrm{Ent}(P_t^\ast\mu_0|\mu_\infty)\leq \bar{c}\e^{-2\lambda t}\mathrm{Ent}(\mu_0|\mu_\infty),\ \ t\geq 1, \mu_0\in\scr P_2(\R^d)
\end{align}
for some constant $\bar{c}>0$.
\end{enumerate}
\vspace{-1.5em}
\begin{rem} As explained in the introduction, to derive the uniform in time log-Sobolev inequality, \cite[Proposition 3.9]{MRW} requires .
$$-\left\<\nabla U(x)+\nabla_xW(x,z)-\nabla U(y)-\nabla_yW(y,z),x-y\right\>\leq -\rho|x-y|^2,\ \ |x|\vee|y|\geq R,$$
which implies
\begin{align*}&-\left\<\nabla U(x)+\int_{\R^d}\nabla_xW(x,z)\mu(\d z)-\nabla U(y)-\int_{\R^d}\nabla_yW(y,z)\mu(\d z),x-y\right\>\\
&\leq -\rho|x-y|^2,\ \ |x-y|\geq 2R.
\end{align*}
So, the condition on the drift in \cite[Proposition 3.9]{MRW} is slightly stronger than {\bf (H)}. Besides, we do not require the high diffusivity as in \cite[Proposition 3.9]{MRW}.
Moreover, the
assumption on $\sigma$ in Theorem \ref{Erg}(3) is to derive the log-Harnack inequality. To our best knowledge, the log-Harnack inequality for McKean-Vlasov SDEs with $\sigma$ being dependent on the measure variable and $b$ satisfying \eqref{stmon} is still open.
\end{rem}
\begin{rem} Compared with the assumption $\sup_{\mu\in\scr P_2(\R^d)}|b(0,\mu)|<\infty$ in \cite[Corollary 2.14]{ZSQ}, $|b(0,\mu)|$ can be of linear growth in $\mu$ in Theorem \ref{Erg}, see Example \ref{exabc} below. The condition of $K_I$ being small enough is essential, otherwise the uniqueness of invariant probability measures may fail, see for instance \cite{Daw}.
\end{rem}
\end{thm}

\begin{proof}
(1) Recall $\Phi(\mu)$ is given in Theorem \ref{thm: IPM}. 
Combining \eqref{losob0} and \cite[Theorem 2.6(2)]{W2020}, there exist constants $c_1,\lambda_1>0$ independent of $\mu$ such that
\begin{align} \label{RR3}
\W_2((\tilde{P}_t^\mu)^\ast\nu,\Phi(\mu))\leq c_1\e^{-\lambda_1 t}\W_2(\nu,\Phi(\mu)),\ \ \mu\in\hat{\scr P}_{2,M}(\R^d), \nu\in\scr P_2(\R^d), t\geq 0.
\end{align}
Define $$\delta_0:=\left(\inf_{t>\frac{\log c_1}{\lambda_1}}\frac{\sqrt{\frac{1}{K_1}(\e^{K_1t}-1)}}{1-c_1\e^{-\lambda_1 t}
}\right)^{-2}.$$
By taking advantage of \cite[Proposition 2.2]{ZSQ} and \eqref{Mon1}, when $K_I<\min(K_2,\delta_0)$, \eqref{EKy} has a unique invariant probability measure $\mu_\infty\in\hat{\scr P}_{2,M}(\R^d)$. For reader's convenience, we give more details.
For $\mu_1,\mu_2\in\scr P_2(\R^d)$, let
$\tilde{X}_t^i,i=1,2$ solve
\begin{align*}
\d \tilde{X}_{t}^{i}=b(\tilde{X}_{t}^{i},\mu_i)\d t+\sigma(\mu_i)\d W_t
\end{align*}
with $\tilde{X}_0^1=\tilde{X}_0^2$ satisfying $\L_{\tilde{X}_0^1}=\Phi(\mu_1)$. Then it holds $\L_{\tilde{X}_t^1}=\Phi(\mu_1)$ and $\L_{\tilde{X}_t^2}=(\tilde{P}_t^{\mu_2})^\ast(\Phi(\mu_1))$. It follows from \eqref{Mon1} and Gronwall's inequality that
\begin{align}\label{kgy}
\W_2((\tilde{P}_t^{\mu_2})^\ast(\Phi(\mu_1)),\Phi(\mu_1))^2\leq \E|\tilde{X}_{t}^{1}-\tilde{X}_{t}^{2}|^2\leq \int_0^tK_I\e^{K_1(t-s)}\W_2(\mu_1,\mu_2)^2\d s,
\end{align}
which together with \eqref{RR3} for $\mu=\mu_2,\nu=\Phi(\mu_1)$ implies
\begin{align*}
\W_2(\Phi(\mu_1),\Phi(\mu_2))\leq c_1\e^{-\lambda_1 t}\W_2(\Phi(\mu_1),\Phi(\mu_2))+\sqrt{K_I\frac{1}{K_1}(\e^{K_1t}-1)}\W_2(\mu_1,\mu_2).
\end{align*}
Hence, combining this with Theorem \ref{logta0} implies that when $K_I<\min\{K_2,\delta_0\}$, $\Phi$ is a contractive map on $\hat{\scr P}_{2,M}(\R^d)$ and has a unique fixed point $\mu_\infty$ satisfying $\Phi(\mu_\infty)=\mu_\infty$, which means that \eqref{EKy} has a unique IPM $\mu_\infty\in \hat{\scr P}_{2,M}(\R^d)$. So, the log-Sobolev inequality \eqref{losob01} follows from \eqref{losob0}.
%

(2) Next, we prove the exponential ergodicity in $\W_2$.
Let
$X_t^i,i=1,2$ solve
\begin{align*}\
\d X_{t}^{1}=b(X_{t}^{1},P_t^\ast\mu_0)\d t+\sigma(P_t^\ast\mu_0)\d W_t,
\end{align*}
and
\begin{align*}
\d X_{t}^{2}=b(X_{t}^{2},\mu_\infty)\d t+\sigma(\mu_\infty)\d W_t
\end{align*}
with $X_0^1=X_0^2$ satisfying $\L_{X_0^1}=\mu_0$. Similar to \eqref{kgy}, It\^{o}'s formula and \eqref{Mon1} imply that
\begin{align*}\W_2(P_t^\ast\mu_0,(\tilde{P}_t^{\mu_\infty})^\ast\mu_0)^2&\leq \E|X_{t}^{1}-X_{t}^{2}|^2\\
&\leq\int_0^t\e^{K_1(t-s)}K_I\W_2(P_s^\ast\mu_0,\mu_\infty)^2\d s\leq \e^{K_1t}K_I\int_0^t\W_2(P_s^\ast\mu_0,\mu_\infty)^2\d s.
\end{align*}
This, together with \eqref{RR3} for $\mu=\mu_\infty,\nu=\mu_0$,  yields that
\begin{align*}\W_2(P_t^\ast\mu_0,\mu_\infty)^2&
\leq 2\W_2(P_t^\ast\mu_0,(\tilde{P}_t^{\mu_\infty})^\ast\mu_0)^2+2\W_2((\tilde{P}_t^{\mu_\infty})^\ast\mu_0,\mu_\infty)^2\\
&\leq 2\e^{K_1t}K_I\int_0^t \W_2(P_s^\ast\mu_0,\mu_\infty)^2\d s +2c_1^2\e^{-2\lambda_1 t}\W_2(\mu_0,\mu_\infty)^2.
\end{align*}
Gronwall's inequality yields
\begin{align}\label{w2c}
\nonumber&\W_2(P_t^\ast\mu_0,\mu_\infty)^2\\
&\leq \left(\int_0^t2c_1^2\e^{-2\lambda_1 s}2\e^{K_1s}K_I\e^{2\e^{K_1t}K_I(t-s)}\d s +2c_1^2\e^{-2\lambda_1 t}\right)\W_2(\mu_0,\mu_\infty)^2.
\end{align}
Take
$$\delta_1=\sup\left\{K_I>0:\inf_{t>\frac{\log(2c_1^2)}{2\lambda_1}}\int_0^t2c_1^2\e^{-2\lambda_1 s}2\e^{K_1s}K_I\e^{2\e^{K_1t}K_I(t-s)}\d s +2c_1^2\e^{-2\lambda_1 t}<1\right\}.$$
So, when $K_I<\min(K_2,\delta_0,\delta_1)$, we can choose $\hat{t}>0$ and $\alpha\in(0,1)$ such that
\begin{align}\label{CMYke}\W_2(P_{\hat{t}}^\ast\mu_0,\mu_\infty)\leq \alpha\W_2(\mu_0,\mu_\infty).
\end{align}
By the semigroup property $P_{t+s}^\ast= P_{t}^\ast P_{s}^\ast$ and \eqref{w2c}-\eqref{CMYke}, we get \eqref{Wcove}.

(3) Finally, assuming $\sigma(x,\gamma)=\sigma$ and \eqref{stmon}, the log-Harnack inequality holds by \cite{FYW1}, i.e.
\begin{align}\label{ktl}
\nonumber&(P_t^\ast\mu_0) (\log f)\\
&\leq \log (P_t^\ast\nu_0) (f)+\eta(t)\W_2(\mu_0,\nu_0)^2, \ \ t>0,\mu_0,\nu_0\in\scr P_2(\R^{d}), 1<f\in\scr B_b(\R^d)
\end{align}
holds for some measurable function $\eta:(0,\infty)\to(0,\infty)$.
Moreover, the Talagrand inequality also holds due to the log-Sobolev inequality \eqref{losob01}, see for instance \cite{OV}, which together with \eqref{ktl} and \eqref{Wcove} implies \eqref{Wcoveent}, see for instance \cite{RW} for more details.
\end{proof}

\vspace{0.5em}
\begin{exa}\label{exabc} Let $V(x)=-|x|^4+|x|^2, x\in\R^d$, $W\in C^2(\R^d;\R)$, $U\in C^2(\R^d;\R^d)$ with $\nabla U$ being bounded and there exists a constant $K_I>0$ such that
$$|\nabla W(x)-\nabla W(y)|+|\nabla U(x)-\nabla U(y)|\leq K_I|x-y|,\ \ x,y\in\R^d.$$
 Set
\begin{align*}
b(x,\mu)=\nabla V(x)+\int_{\R^d}\nabla W(x-y)\mu(\d y)~ \mbox{ and } ~ \sigma(\mu)=\sqrt{I_{d\times d}+\int_{\R^d}(\nabla U(\nabla U)^\ast)(y)\mu(\d y)}.
\end{align*}
Then {\bf(H)} holds and if $K_I$ is small enough, the assertions in Theorem \ref{Erg} hold.
\end{exa}

\vspace{0.2em}
\begin{rem} If $\sigma$ also depends on the spatial variable and the Bakry-\'{E}mery curvature condition holds, i.e.
\begin{align}\label{BEC}\Gamma_2^\mu(f)\geq K|\sigma^\ast\nabla f|^2,\ \ f\in C_0^\infty(\R^d),\mu\in\scr P_2(\R^d)
\end{align}
for some constant $K\in\R$,
where $$L^\mu=\<b(\cdot,\mu), \nabla \>+\frac{1}{2}\mathrm{tr}((\sigma\sigma^\ast)(\cdot,\mu)\nabla^2),$$
and $$\Gamma_2^\mu(f):=\frac{1}{2}L^\mu(|\sigma^\ast\nabla f|^2)-\<\sigma^\ast\nabla f, \sigma^\ast\nabla L^\mu f\>,$$
then \eqref{tillo} is true, see for instance \cite{Bakry,Hsu}.
Hence, the assertions in Theorem \ref{logta0} and Theorem \ref{Erg} also hold by replacing $\sigma(x,\gamma)=\sigma(\gamma)$ with \eqref{BEC}. However, \eqref{BEC} is not easy to check.
\end{rem}
\section{Degenerate case}
In this section, we consider
\begin{align}\label{Deg}\left\{
  \begin{array}{ll}
    \d X_t=Y_t\d t,  \\
    \d Y_t=b_t(X_t,Y_t,\L_{(X_t,Y_t)})\d t+\sigma_t(\L_{(X_t,Y_t)})\d W_t,
  \end{array}
\right.
\end{align}
where $b:[0,\infty)\times\R^{2d}\times \scr P(\R^{2d})\to\R^d$, $\sigma:[0,\infty)\times\scr P(\R^{2d})\to\R^d\otimes\R^n$ are measurable and bounded on bounded sets and $\{W_t\}_{t\geq 0}$ is an $n$-dimensional standard Brownian motion on some complete filtered probability space $(\Omega, \scr F, (\scr F_t)_{t\geq 0},\P)$. As said above, when \eqref{Deg} is well-posed in $\scr P_2(\R^{2d})$, we use $P_t^\ast \mu_0$ to denote the distribution of the solution to \eqref{Deg} with initial distribution $\mu_0\in\scr P_2(\R^{2d})$. For any $\mu_\cdot\in C([0,\infty),\scr P_2(\R^{2d}))$, we also consider the time-inhomogeneous decoupled SDEs:
\begin{align}\label{Degde}\left\{
  \begin{array}{ll}
    \d X_t^\mu=Y_t^\mu\d t,  \\
    \d Y_t^\mu=b_t(X_t^\mu,Y_t^\mu,\mu_t)\d t+\sigma_t(\mu_t)\d W_t.
  \end{array}
\right.
\end{align}
We make the following assumptions
\begin{enumerate}
\item[{\bf(B)}]
$b$ is bounded on bounded sets of $[0,\infty)\times \R^{2d}\times \scr P_2(\R^{2d})$. For any $t\geq 0$ and $\gamma\in\scr P_2(\R^{2d})$, $b_t(\cdot,\gamma)$ is continuous in $\R^{2d}$. There exist  constants $K_1,\tilde{K}_I>0$ such that for all $t\geq 0$, $x,\tilde{x},y,\tilde{y}\in\R^d$ and $\gamma,\tilde\gamma\in\scr P_2(\R^{2d})$,
\begin{align*}&\nonumber2\<b_t(x,y,\gamma)-b_t(\tilde{x},\tilde{y},\tilde\gamma),y-\tilde{y}\> +2\<y-\tilde{y},x-\tilde{x}\>+\|\sigma_t(\gamma)-\sigma_t(\tilde\gamma)\|_{HS}^2\\
&\leq K_1(|x-\tilde{x}|^2+|y-\tilde{y}|^2)+\tilde{K}_I\W_2(\gamma,\tilde\gamma)^2,
\end{align*}
and
$$\|\sigma_t(\gamma)\|_{HS}^2\leq K_1(1+\gamma(|\cdot|^2)).$$
\end{enumerate}
 Let
$$B_t(x,y,\gamma)=\left(
             \begin{array}{c}
               y \\
               b_t(x,y,\gamma) \\
             \end{array}
           \right), \ \ \Sigma_t(\gamma)=\left(
                                           \begin{array}{c}
                                             0_{d\times n} \\
                                             \sigma_t(\gamma) \\
                                           \end{array}
                                         \right),\ \ x,y\in\R^d,\gamma\in\scr P(\R^{2d}).
$$
Let $Z_t=\left(
               \begin{array}{c}
                 X_t \\
                 Y_t \\
               \end{array}
             \right)
$ and $Z_t^\mu=\left(
               \begin{array}{c}
                 X_t^\mu \\
                 Y_t^\mu \\
               \end{array}
             \right)
$. Then \eqref{Deg} and \eqref{Degde} can be rewritten as
$$\d Z_t=B_t(Z_t,\L_{Z_t})+\Sigma_t(\L_{Z_t})\d W_t$$
and
$$\d Z_t^\mu=B_t(Z_t^\mu,\mu_t)+\Sigma_t(\mu_t)\d W_t.$$

So, under {\bf(B)}, Theorem \ref{wel} implies that \eqref{Deg} is well-posed in $\scr P_2(\R^{2d})$ and \eqref{Degde} is also well-posed.
Let $P_{s,t}^\mu$ be the semigroup associated to \eqref{Degde}.

For a differentiable function $f$ on $\R^d\times\R^d$, $i=1,2$, let $\nabla^{(i)}f(x_1,x_2)$ denote the gradient of $f$ along $x_i$.
\subsection{Log-Sobolev inequality for $P_{s,t}^\mu$ and $P_t^\ast\mu_0$}

The following result is a corollary of Theorem \ref{Poi3}.
\begin{thm}\label{Poi23} Assume {\bf(B)} and there exists a constant $\delta_1>0$ such that
\begin{align}\label{upb}\sigma_t(\gamma)\sigma_t(\gamma)^\ast\leq \delta_1I_{d\times d},\ \ \gamma\in\scr P_2(\R^{2d}), t\geq 0.
\end{align}
Then the assertions in Theorem \ref{Poi3} hold for \eqref{Deg} and \eqref{Degde} on $\R^{2d}$ replacing \eqref{E1} and \eqref{E2} on $\R^{d}$.
\end{thm}
\begin{rem} It is not difficult to see that the log-Sobolev inequality \eqref{til} does not hold if $\nabla $ is replaced by $\nabla^{(2)}$. This is essentially different from the non-degenerate case. In fact, the inequality
\begin{align*}|\nabla P_{s,t}^\mu f|\leq \e^{\frac{K_1}{2}(t-s)}P_{s,t}^\mu |\nabla f|, \ \ f\in C_b^\infty(\R^{2d}), 0\leq s\leq t
\end{align*}
does not hold if $\nabla $ is replaced by $\nabla^{(2)}$.
\end{rem}
\subsection{Time-homogeneous decoupled SDEs}

From now on, we consider the time-homogeneous case, i.e. consider
\begin{align}\label{Degth}\left\{
  \begin{array}{ll}
    \d X_t=Y_t\d t,  \\
    \d Y_t=b(X_t,Y_t,\L_{(X_t,Y_t)})\d t+\sigma(\L_{(X_t,Y_t)})\d W_t.
  \end{array}
\right.
\end{align}
For any $\mu\in\scr P_2(\R^{2d})$, we first study time-homogeneous decoupled SDEs
\begin{align}\label{tih12}\left\{
  \begin{array}{ll}
    \d \tilde{X}_t^\mu=\tilde{Y}_t^\mu\d t,  \\
    \d \tilde{Y}_t^\mu=b(\tilde{X}_t^\mu,\tilde{Y}_t^\mu,\mu)\d t+\sigma(\mu)\d W_t.
  \end{array}
\right.
\end{align}
Let $\tilde{P}_t^\mu$ be the associated semigroup to \eqref{tih12} and $(\tilde{P}_t^\mu)^\ast\gamma$ stand for the distribution of the solution to \eqref{tih12} with initial distribution $\gamma\in\scr P_2(\R^{2d})$.
To derive the hyperboundedness for time-homogenenous decoupled SDE \eqref{tih12}, we will adopt Wang's Harnack inequality. Note that in the present degenerate case, the Lipschitz continuity of $b$ in the spatial variable is required to ensure Wang's Harnack inequality so that we present the following condition.
\begin{enumerate}
\item[{\bf(C)}] There exist constants $K_M>0,K_I>0$ such that for all $z,\bar{z}\in\R^{2d}, \gamma,\tilde{\gamma}\in\scr P_2(\R^{2d})$,
\begin{align}\label{Lipsc}|b(z,\gamma)-b(\bar{z},\tilde{\gamma})|+\|\sigma(\gamma)-\sigma(\tilde{\gamma})\|_{HS} \leq K_M|z-\bar{z}|+K_I\W_2(\gamma,\tilde{\gamma}).
\end{align}
Moreover, there exist constants $\theta,r,R\geq 0, $ and $r_0\in(-1,1)$ such that for any $x,y,\bar{x},\bar{y}\in\R^{d}$ satisfying $|x-\bar{x}|^2+|y-\bar{y}|^2\geq R^2$ and $\mu\in\scr P_2(\R^{2d})$,
\begin{align}\label{Patdi}
\nonumber&\<r^2(x-\bar{x})+rr_0(y-\bar{y}),y-\bar{y}\>\\
&+\<(y-\bar{y})+rr_0(x-\bar{x}),b(x,y,\mu)-b(\bar{x},\bar{y},\mu)\>\\
\nonumber&\leq -\theta(|x-\bar{x}|^2+|y-\bar{y}|^2).
\end{align}
\end{enumerate}
\eqref{Patdi} with $R=0$ means uniform dissipativity while \eqref{Patdi} with $R>0$ becomes partial dissipativity.
Let $$\psi((x,y),(\bar{x},\bar{y}))=\sqrt{r^2|x-\bar{x}|^2/2+|y-\bar{y}|^2/2+rr_0\<x-\bar{x},y-\bar{y}\>},\ \ (x,y),(\bar{x},\bar{y})\in\R^{2d}.$$
Then there exists a constant $C_\psi>1$ depending on $r,r_0$ such that for any $(x,y),(\bar{x},\bar{y})\in\R^{2d}$,
\begin{align}\label{cpsid}C_\psi^{-1}(|x-\bar{x}|^2+|y-\bar{y}|^2)\leq \psi((x,y),(\bar{x},\bar{y}))^2\leq C_\psi(|x-\bar{x}|^2+|y-\bar{y}|^2).
\end{align}

\subsubsection{Invariant probability measure}
 In this part, we investigate the existence, uniqueness and concentration of the invariant probability measure to \eqref{tih12}.
\begin{thm}\label{fixdi} Assume {\bf (C)} and there is a constat $\delta_2>0$ such that \begin{align}\label{lowbo}\sigma(\gamma)\sigma^\ast(\gamma)\geq \delta_2I_{d\times d},\ \ \gamma\in\scr P_2(\R^{2d}).
\end{align}
\begin{enumerate}
\item[$(1)$]
Then for any $\mu\in\scr P_2(\R^{2d})$, \eqref{tih12} has a unique invariant probability measure denoted by $\Phi(\mu)$.
\item[$(2)$] Moreover, there exists a constant $\eta_0>0$ such that when $K_I< \eta_0$, we can find a constant $N>0$ such that
$\Phi$ is a map on $\hat{\scr P}_{2,N}(\R^{2d})$ with
$$\hat{\scr P}_{2,N}(\R^{2d}):=\left\{\mu\in\scr P_2(\R^{2d}): \|\mu\|_2^2\leq  N \right\}.$$
\item[$(3)$] In addition, there exist constants $\vv>0,C_0>0$ such that \begin{align}\label{emo3}(\Phi(\mu))(\e^{\vv|\cdot|^2})<C_0, \ \ \mu\in\hat{\scr P}_{2,N}(\R^{2d}).
\end{align}
\end{enumerate}
\end{thm}
\begin{proof}
(1) Firstly, it follows from \eqref{Lipsc} that
\begin{align}\label{gromu}
|b(0,\mu)|+\|\sigma(\mu)\|_{HS}\leq K_I\|\mu\|_2+|b(0,\delta_0)|+\|\sigma(\delta_0)\|_{HS},\ \ \mu\in\scr P_2(\R^{2d}).
\end{align}
Let $$\tilde{\psi}(x,y)=\psi((x,y),0)=\sqrt{r^2|x|^2/2+|y|^2/2+rr_0\<x,y\>},\ \ (x,y)\in\R^{2d}.$$
It follows from It\^{o}'s formula, \eqref{Patdi} and \eqref{gromu} that for some positive constants $c_1,c_2,c_3$,
\begin{align}\label{sec3}\nonumber\d \tilde{\psi}(\tilde{X}_t^\mu, \tilde{Y}_t^\mu)^2&=\<r^2\tilde{X}_t^\mu+rr_0\tilde{Y}_t^\mu,\tilde{Y}_t^\mu\>\d t+\<\tilde{Y}_t^\mu+rr_0\tilde{X}_t^\mu,b(\tilde{X}_t^\mu,\tilde{Y}_t^\mu,\mu)\>\d t\\
\nonumber&\qquad\quad+\frac{1}{2}\|\sigma(\mu)\|_{HS}^2\d t+\<\tilde{Y}_t^\mu+rr_0\tilde{X}_t^\mu,\sigma(\mu)\d W_t\>\\
&\leq -\theta(|\tilde{X}_t^\mu|^2+|\tilde{Y}_t^\mu|^2)\d t+\left(\theta+\frac{r^2}{2}+|rr_0|+K_M(1+|rr_0|)\right)R^2\d t\\
&\nonumber\qquad\quad+\<\tilde{Y}_t^\mu+rr_0\tilde{X}_t^\mu,b(0,\mu)\>\d t+\frac{1}{2}\|\sigma(\mu)\|_{HS}^2\d t+\<\tilde{Y}_t^\mu+rr_0\tilde{X}_t^\mu,\sigma(\mu)\d W_t\>\\
\nonumber&\leq -c_2(|\tilde{X}_t^\mu|^2+|\tilde{Y}_t^\mu|^2)\d t+c_1\d t+c_3K_I^2\|\mu\|_2^2\d t+\<\tilde{Y}_t^\mu+rr_0\tilde{X}_t^\mu,\sigma(\mu)\d W_t\>.
\end{align}
So, we arrive at
\begin{align}\label{Emp3}
\frac{1}{t}\int_0^t\|(\tilde{P}_s^\mu)^\ast\delta_0\|_2^2\d s\leq  (c_1+c_3K_I^2\|\mu\|_2^2)/{c_2}.
\end{align}
By the standard argument of tightness, we derive from \eqref{Emp3} that \eqref{tih12} has an invariant probability measure  denoted by $\Phi(\mu)$, which satisfies
\begin{align}\label{Pit3}\|\Phi(\mu)\|_2^2\leq (c_1+c_3K_I^2\|\mu\|_2^2)/{c_2}.
\end{align}
Moreover, by \eqref{lowbo}, \eqref{Lipsc} and \cite[Lemma 3.2]{FYW2017}, Wang's Harnack inequality for $\tilde{P}_t^\mu$ with power $2$ holds, i.e. for any $t>0$, there exists a constant $c_0>0$ independent of $\mu$ such that
\begin{align}\label{Hap}
(\tilde{P}_t^\mu f)^2(z)\leq \tilde{P}_t^\mu f^2(\bar{z})\e^{c_0|z-\bar{z}|^2}, \ \ f\in\scr B_b^+(\R^{2d}), z,\bar{z}\in\R^{2d}, \mu\in\scr P_2(\R^{2d}).
\end{align}
This implies the uniqueness of the invariant probability measure of \eqref{tih12}, see for instance \cite[Theorem 1.4.1(3)]{Wbook}.

(2) Next let $K_I<\sqrt{\frac{c_2}{c_3}}$. For any $\mu\in\scr P_2(\R^{2d})$ with $\|\mu\|_2^2\leq c_1(c_2-c_3K_I^2)^{-1}$, it follows from \eqref{Pit3} that
\begin{align*}\|\Phi(\mu)\|_2^2\leq  (c_1+c_3K_I^2\|\mu\|_2^2)/{c_2}\leq  c_1(c_2-c_3K_I^2)^{-1}.
\end{align*}
Let $N=c_1(c_2-c_3K_I^2)^{-1}$. This means that $\Phi$ is a map on $\hat{\scr P}_{2,N}(\R^{2d})$ with
$$\hat{\scr P}_{2,N}(\R^{2d}):=\left\{\mu\in\scr P_2(\R^{2d}): \|\mu\|_2^2\leq N\right\}.$$

(3) Finally, we study the concentration of $\Phi(\mu)$ for any $\mu\in\hat{\scr P}_{2,N}(\R^{2d})$.
By \eqref{sec3}, \eqref{gromu} and It\^{o}'s formula for $\e^{\vv\tilde{\psi}(\tilde{X}_t^\mu,\tilde{Y}_t^\mu)^2}$, for any $\mu\in\hat{\scr P}_{2,N}(\R^{2d})$, we derive
\begin{align*}&\d \e^{\vv\tilde{\psi}(\tilde{X}_t^\mu,\tilde{Y}_t^\mu)^2}\\
&\leq\vv\e^{\vv\tilde{\psi}(\tilde{X}_t^\mu,\tilde{Y}_t^\mu)^2}\left(-c_2(|\tilde{X}_t^\mu|^2+|\tilde{Y}_t^\mu|^2)\d t+c_1\d t+c_3K_I^2\|\mu\|_2^2\d t+\<\tilde{Y}_t^\mu+rr_0\tilde{X}_t^\mu,\sigma(\mu)\d W_t\>\right)\\
&\quad+\frac{1}{2}\vv^2\e^{\vv\tilde{\psi}(\tilde{X}_t^\mu,\tilde{Y}_t^\mu)^2}| \sigma(\mu)^\ast[\tilde{Y}_t^\mu+rr_0\tilde{X}_t^\mu]|^2\\
&\leq\vv\e^{\vv\tilde{\psi}(\tilde{X}_t^\mu,\tilde{Y}_t^\mu)^2}\left((-c_2+\vv c_4)(|\tilde{X}_t^\mu|^2+|\tilde{Y}_t^\mu|^2)\d t+c_5\d t+\<\tilde{Y}_t^\mu+rr_0\tilde{X}_t^\mu,\sigma(\mu)\d W_t\>\right)
\end{align*}
for some positive constants $c_4,c_5$.
When $\vv>0$ is small enough, we can find constants $c_6(\vv), c_7(\vv)>0$ such that
\begin{align*} &\vv\e^{\vv\tilde{\psi}(\tilde{X}_t^\mu,\tilde{Y}_t^\mu)^2}\left((-c_2+\vv c_4)(|\tilde{X}_t^\mu|^2+|\tilde{Y}_t^\mu|^2)+c_5\right)\leq c_6(\vv)-c_7(\vv)\e^{\vv\tilde{\psi}(\tilde{X}_t^\mu,\tilde{Y}_t^\mu)^2}.
\end{align*}
This together with \eqref{cpsid} yields
$\frac{1}{t}\int_0^t \big(\tilde{P}_s^\mu)^\ast \delta_0\big)wo fangni (\e^{\vv C_{\psi}^{-1} |\cdot|^2})\d s
\leq \frac{c_6(\vv)t+1}{c_7(\vv)t}.
$
Thus, \eqref{emo3} is derived by following the line to derive
 \eqref{Pit3}.
\end{proof}

\subsubsection{Wang's Harnack inequality and hyperboundedness}
The Harnack inequality \eqref{Hap} is not as sharp as the one in \eqref{niHar}. In the following theorem, we will first derive a sharper Wang's Harnack inequality, which implies the hyperboundedness of $\tilde{P}_t^{\mu}$ for large enough $t$. Note that the proof of \eqref{Hap13} is different from that of \eqref{niHar}.

\begin{thm}Assume  {\bf(C)} and \eqref{lowbo}.
Then \begin{align}\label{Hap13}
\nonumber&(\tilde{P}_t^\mu f)^2(z)
\leq \tilde{P}_t^\mu f^2(\bar{z})\\
&\quad\times\e^{c_0C_1\frac{C_\psi^2}{\theta}+c_0C_\psi^2\e^{-\frac{\theta}{C_\psi} (t-1)}|z-\bar{z}|^2}, \ \ t\geq 1, f\in\scr B_b^+(\R^{2d}), z,\bar{z}\in\R^{2d}, \mu\in\scr P_2(\R^{2d})
\end{align}
holds for some constant $C_1>0$ depending on $r,r_0,\theta,R, K_M$.
Consequently, if $K_I$ is small enough,  then there exist constants $t_0>0,\tilde{C}>0$ such that
\begin{align}\label{hyper3}
\|\tilde{P}_{t_0}^\mu\|_{L^2(\Phi(\mu))\to L^{4}(\Phi(\mu))}\leq \tilde{C}, \ \ \mu\in\hat{\scr P}_{2,N}(\R^{2d}).
\end{align}
\end{thm}
\begin{proof}
Let $(X_t,Y_t)$ and $(\bar{X}_t,\bar{Y}_t)$ solve \eqref{tih12} with $(X_0,Y_0)=z,(\bar{X}_0,\bar{Y}_0)=\bar{z}$.
By It\^{o}'s formula, \eqref{Lipsc}, \eqref{Patdi} and \eqref{cpsid}, we can find a constant $C_1>0$ depending on $r,r_0,\theta,R, K_M$ such that
\begin{align*}
\d \psi((X_t,Y_t),(\bar{X}_t,\bar{Y}_t))^2
&=\<r^2(X_t-\bar{X}_t)+rr_0(Y_t-\bar{Y}_t),Y_t-\bar{Y}_t\>\\
&+\<(Y_t-\bar{Y}_t)+rr_0(X_t-\bar{X}_t),b(X_t,Y_t,\mu)-b(\bar{X}_t,\bar{Y}_t,\mu)\>\\
&\leq -\theta(|X_t-\bar{X}_t|^2+|Y_t-\bar{Y}_t|^2)\d t+C_1\d t\\
&\leq -\frac{\theta}{C_\psi}(|X_t-\bar{X}_t|^2+|Y_t-\bar{Y}_t|^2)\d t+C_1\d t.
\end{align*}
This again combined with \eqref{cpsid} implies that
\begin{align}\label{difxy}
|X_t-\bar{X}_t|^2+|Y_t-\bar{Y}_t|^2\leq C_1\frac{C_\psi^2}{\theta}+C_\psi^2\e^{-\frac{\theta}{C_\psi} t}|z-\bar{z}|^2.
\end{align}
On the other hand, it follows from \eqref{Hap} with $t=1$ and \eqref{difxy} that
\begin{align}\label{Hap12}
\nonumber(\tilde{P}_1^\mu f)^2(X_{t-1})&\leq \tilde{P}_1^\mu f^2(\bar{X}_{t-1})\e^{c_0(|X_{t-1}-\bar{X}_{t-1}|^2+|Y_{t-1}-\bar{Y}_{t-1}|^2)}\\
&\leq\tilde{P}_1^\mu f^2(\bar{X}_{t-1})\e^{c_0C_1\frac{C_\psi^2}{\theta}+c_0C_\psi^2\e^{-\frac{\theta}{C_\psi} (t-1)}|z-\bar{z}|^2},\ \ t\geq 1, f\in\scr B_b^+(\R^{2d}).
\end{align}
Taking expectation in \eqref{Hap12} and using the semigroup property and Jensen's inequality, we deduce
\eqref{Hap13}.
Repeating the same argument to derive \eqref{hyper} from \eqref{niHar} and \eqref{emo}, we gain \eqref{hyper3} from \eqref{emo3} and \eqref{Hap13}.
\end{proof}
\subsubsection{Log-Sobolev inequality}
\begin{thm}\label{logta}
Assume {\bf (C)}, \eqref{upb}, \eqref{lowbo} and $K_I$ is small enough.
\begin{enumerate}
\item[$(1)$] Then for any $\mu\in\hat{\scr P}_{2,N}(\R^{2d})$, $\Phi(\mu)$ satisfies the log-Sobolev inequality
\begin{align}\label{losob1}\nonumber&\Phi(\mu)(f\log f)-\Phi(\mu)(f)\log \Phi(\mu)(f)\\
&\leq C_{L}(\Phi(\mu))\Phi(\mu)(|\nabla f^{\frac{1}{2}}|^2),\ \ 0<f\in C_0^\infty(\R^{2d})
\end{align}
for some constant $C_L(\Phi(\mu))>0$ depending on $\Phi(\mu)$. Consequently,  for any $\mu\in\hat{\scr P}_{2,N}(\R^{2d})$, the Talagrand inequality
 \begin{align}\label{Tal}\W_2(\mu_0,\Phi(\mu))^2\leq  \tilde{C}_{L}(\Phi(\mu))\mathrm{Ent}(\mu_0|\Phi(\mu)),\ \ \mu_0\in\scr P_2(\R^{2d}).
\end{align}
holds for some constant $\tilde{C}_L(\Phi(\mu))>0$.
\item[$(2)$]
If moreover, for any $R>0$ and $N>0$, there exist positive constants $C_1$ and $C_2$ depending on $R$, $N$ such that
\begin{align}\label{upbloh}C_1\leq \rho_\mu(x)\leq C_2,\ \ x\in B_R, \mu\in\hat{\scr P}_{2,N}(\R^{2d})
\end{align}
for $\rho_\mu:=\frac{\d \Phi(\mu)}{\d x}$.  Then there exists a constant $\bar{C}>0$ depending on $N$
such that
\begin{align}\label{losob1h}&\nonumber\Phi(\mu)(f\log f)-\Phi(\mu)(f)\log \Phi(\mu)(f)\\
&\leq \bar{C}\Phi(\mu)(|\nabla f^{\frac{1}{2}}|^2),\ \ 0<f\in C_0^\infty(\R^{2d}), \mu\in\hat{\scr P}_{2,N}(\R^{2d}),
\end{align}
and
 \begin{align}\label{Talht}\W_2(\mu_0,\Phi(\mu))^2\leq  \bar{C}\mathrm{Ent}(\mu_0|\Phi(\mu)),\ \ \mu_0\in\scr P_2(\R^{2d}), \mu\in\hat{\scr P}_{2,N}(\R^{2d}).
\end{align}
\end{enumerate}
\end{thm}
\begin{rem}
	The drift $b$ in Theorem \ref{logta} is allowed to be partially dissipative, which improves the existing results in \cite{RW} where $b$ is assumed to be uniformly dissipative. In contrast to the results in \cite{Villani} and \cite{Baudoin}, the drift is allowed to be not of gradient-type.
\end{rem}
\begin{proof}
(1)	Combining the hyperboundedness \eqref{hyper3} and the log-Sobolev inequality \eqref{til} with $s=0$ due to Theorem \ref{Poi23} and $\tilde{P}_t^\mu$ replacing $P_{s,t}^\mu$, we obtain a defective log-Sobolev inequality by Lemma \ref{RWlog}(1), i.e. there exist constant $\tilde{C}_1,\tilde{C}_2>0$ independent of $\mu$ such that
\begin{align}\label{delogso}\nonumber&\Phi(\mu)(f\log f)-\Phi(\mu)(f)\log \Phi(\mu)(f)\\
&\leq \tilde{C}_1\Phi(\mu) ( |\nabla f^{\frac{1}{2}}|^2 )+\tilde{C}_2,\ \ 0<f\in C_0^\infty(\R^d),\Phi(\mu)(f)=1, \mu\in\hat{\scr P}_{2,N}(\R^{2d}).
\end{align}
 The irreducibility of $\scr E(f,f)=\Phi(\mu)(|\nabla f|^2)$ implies \eqref{losob1} by \cite[Corollary 1.3]{Wang2014}. \eqref{losob1} further implies \eqref{Tal} according to \cite{OV}.

(2) By \eqref{delogso}, \eqref{emo3} and \eqref{upbloh}, the proof is completed by Lemma \ref{RWlog}(2) and \cite{OV}.
\end{proof}
\subsection{Exponential ergodicity of McKean-Vlasov SDEs}
Finally, we will use Theorem \ref{logta} to study the exponential ergodicity of \eqref{Degth}. We divide it into two cases: uniformly dissipative case (\eqref{Patdi} with $R=0$) and partially dissipative case(\eqref{Patdi} with $R>0$). In the former one, $b$ can be general while in the latter one, $b$ is assumed to be of gradient-type.
\subsubsection{Uniformly dissipative case}
\begin{thm}\label{cty}Assume {\bf (C)} with $R=0$, \eqref{upb} and \eqref{lowbo}. If $K_I$ is small enough, then  \eqref{Degth} has a unique IPM $\bar{\mu}\in\scr P_2(\R^{2d})$, and there exist constants $c,\lambda>0 $ such that
\begin{align}\label{w2con}\W_2(P_t^\ast\mu_0,\bar{\mu})\leq c\e^{-\lambda t}\W_2(\mu_0,\bar{\mu}),\ \ t\geq 0, \mu_0\in\scr P_2(\R^{2d}).
\end{align}
Moreover, there exists a constant $\tilde c>0 $ such that
\begin{align}\label{entco} \mathrm{Ent}(P_t^\ast\mu_0|\bar{\mu})\leq \tilde{c}^2\e^{-2\lambda t}\mathrm{Ent}(\mu_0|\bar{\mu}),~ t\geq 1, \mu_0\in\scr P_2(\R^{2d}).
\end{align}
\end{thm}
\begin{rem} \eqref{entco} in Theorem \ref{cty} extends the results in \cite[Theorem 2.4]{RW} and \cite[Theorem 4.2]{HW24}, where $b$ is assumed to be of gradient-type, i.e. $b(x,y,\mu)=-y-x-\nabla V (x,\mu)$.
\end{rem}
\begin{proof} By {\bf (C)} with $R=0$, it follows from \cite[Theorem 4.2]{HW24} that when $K_I$ is small enough, \eqref{Degth} has a unique IPM $\bar{\mu}\in\scr P_2(\R^d)$ satisfying \eqref{w2con}.
Since $\Phi(\bar{\mu})=\bar{\mu}\in\hat{\scr P}_{2,N}(\R^{2d})$, \eqref{Tal} implies the following Talagrand inequality
\begin{align}\label{RR}
\W_2(\mu_0,\bar{\mu})^2\leq  \tilde{C}_{L}\mathrm{Ent}(\mu_0|\bar{\mu}),\ \ \mu_0\in\scr P_2(\R^{2d}).
\end{align}
By  \cite[Theorem 3.1]{HW24}, there exists a measurable function $c\colon (0,\infty)\to(0, \infty)$ such that
\begin{align}\label{lohar}
\nonumber&(P_t^\ast\mu_0) (\log f)\\
&\leq \log (P_t^\ast\nu_0) (f)+c(t)\W_2(\mu_0,\nu_0)^2, \ \ t>0,\mu_0,\nu_0\in\scr P_2(\R^{2d}), 0<f\in\scr B_b(\R^{2d}).
\end{align}
Combining \eqref{lohar}, \eqref{w2con} and \eqref{RR}, we derive \eqref{entco} by the same argument we derived \eqref{Wcoveent}.
\end{proof}
\subsubsection{Partially dissipative case}
\begin{thm}\label{thm: erg deg} Assume \eqref{Patdi} with $R>0$, $\sigma=I_{d\times d}$ and
$$b(x,y,\mu)=-y-\nabla V(x)-\int_{\R^{2d}}\nabla_x W(x,z)\mu(\d z), \ \ x,y\in\R^d$$
with $V:\R^d\to\R$ and $W:\R^{d}\times\R^{2d}\to\R$ satisfying
\begin{align}
\label{Lipvw}\|\nabla^2 V\|+\|\nabla_x^2W\|\leq K_M\quad \mbox{ and } \quad \|\nabla_z\nabla_xW\|\leq K_I
\end{align}
for some constants $K_M,K_I>0$. If $K_I$ is small enough, then  \eqref{Degth} has a unique invariant probability measure $\bar{\mu}\in\scr P_2(\R^{2d})$. Moreover, there exist constants $c>0, \lambda >0$ such that
\begin{align}\label{w3e}\W_2(P_t^\ast\mu_0,\bar{\mu})\leq c\e^{-\lambda t}\W_2(\mu_0,\bar{\mu}), \ \ t\geq 0,
\end{align}
and $$\mathrm{Ent}(P_t^\ast\mu_0|\bar{\mu})\leq c^2\e^{-2\lambda t}\mathrm{Ent}(\mu_0|\bar{\mu}),\ \ t\geq 1.$$
\end{thm}
\begin{rem}\label{vil} In the distribution independent case, \cite[Theorem 39]{Villani} derives the exponential ergodicity in relative entropy by means of the log-Sobolev inequality for the explicit invariant probability measure on the product space $\R^{2d}$. In Theorem \ref{thm: erg deg}, we offer an explicit partially dissipative condition under which the log-Sobolev inequality holds. Moreover, the authors in \cite{GM} adopt the uniform in $N$(number of particles) log-Sobolev inequality and log-Harnack inequality for mean field interacting particle system as well as propagation of chaos to deduce the exponential ergodicity in $L^2$-Wasserstein distance and mean field entropy. Different from \cite{GM}, in Theorem \ref{thm: erg deg}, we derive the exponential ergodicity in relative entropy and instead of starting from the mean field interacting particle system as in \cite{GM}, we begin by investigating the time-homogeneous decoupled SDEs.
\end{rem}
\begin{proof} \eqref{Degth} and \eqref{tih12} reduce to \begin{align}\label{Degth1}\left\{
  \begin{array}{ll}
    \d X_t=Y_t\d t,  \\
    \d Y_t=-Y_t\d t-\nabla V(X_t)\d t-\int_{\R^{2d}}\nabla_x W(X_t,z)\L_{(X_t,Y_t)}(\d z)\d t+\d W_t,
  \end{array}
\right.
\end{align}
and
\begin{align*}\left\{
  \begin{array}{ll}
    \d \tilde{X}_t^\mu=\tilde{Y}_t^\mu\d t,  \\
    \d \tilde{Y}_t^\mu=-\tilde{Y}_t^\mu\d t-\nabla V(\tilde{X}_t^\mu)\d t-\int_{\R^{2d}}\nabla_x W(\tilde{X}_t^\mu,z)\mu(\d z)\d t+\d W_t.
  \end{array}
\right.
\end{align*}
Note that \eqref{Lipvw}   implies \eqref{Lipsc}, which together with \eqref{Patdi} means {\bf(C)} holds.
Observe that the invariant probability measure is
$$\rho_\mu(x,y):=\frac{\d \Phi(\mu)}{\d x\d y}=Z_0\exp\left(-\frac{1}{2}|y|^2-V(x)-\int_{\R^{2d}} (W(x,z)-W(0,z))\mu(\d z)\right),$$
where
$$Z_0=\int_{\R^{2d}}\exp\left(-\frac{1}{2}|y|^2-V(x)-\int_{\R^{2d}} (W(x,z)-W(0,z))\mu(\d z)\right)\d x\d y.$$
By \eqref{Lipvw}, we can find a constant $C_{W,V}>0$ such that
\begin{align*}&|\nabla_xW(x,z)|\leq C_{W,V}(1+|x|+|z|),\\
&|W(x,z)-W(0,z)|\leq C_{W,V}(1+|x|+|z|)|x|,\\
& |V(x)|\leq C_{W,V}(1+|x|^2).
\end{align*}
So, we conclude that for any $R>0$,  there exist constants $C_1$ and $C_2$ depending on $R$, $N$  such that
$$C_1\leq \rho_\mu(x,y)\leq C_2,\ \ (x,y)\in B_R, \mu\in\hat{\scr P}_{2,N}(\R^{2d}).$$
So, by Theorem \ref{logta}, when $K_I$ is small enough, we obtain \eqref{losob1h}, which combined with \cite[Theorem 39]{Villani} implies that  we can find constants $c,\lambda>0$ such that
\begin{align}\label{pty}\mathrm{Ent}((\tilde{P}_t^\mu)^\ast\mu_0|\Phi(\mu))\leq c\e^{-2\lambda t}\mathrm{Ent}(\mu_0|\Phi(\mu)), \ \ t\geq 1, \mu_0\in\scr P_2(\R^{2d}),\mu\in\hat{\scr P}_{2,N}(\R^{2d}).
\end{align}
Note that \eqref{Lipvw} yields that for any $t>0$, there exists a constant $c_0>0$ such that
$$\tilde{P}_t^\mu\log f(z)\leq \log \tilde{P}_t^\mu f(\bar{z})+c_0|z-\bar{z}|^2,\ \ z,\bar{z}\in\R^{2d},f>0,f\in\scr B_b(\R^{2d}),\mu\in\scr P_2(\R^{2d}),$$
see for instance \cite{RW}.
Combining this with \eqref{Talht} and \eqref{pty}, we conclude
\begin{align*}\W_2((\tilde{P}_t^\mu)^\ast\mu_0,\Phi(\mu))^2
&\leq \tilde{C}_L\mathrm{Ent}((\tilde{P}_{t-1}^\mu)^\ast[(\tilde{P}_1^\mu)^\ast\mu_0]|\Phi(\mu))\\
&\leq \tilde{c}_0\e^{-2\lambda t}\mathrm{Ent}((\tilde{P}_1^\mu)^\ast\mu_0|\Phi(\mu))\\
&\leq \tilde{c}_1\e^{-2\lambda t}\W_2(\mu_0,\Phi(\mu))^2,\ \ t\geq 2, \mu_0\in\scr P_2(\R^{2d}), \mu\in\hat{\scr P}_{2,N}(\R^{2d}).
\end{align*}
By \eqref{Lipvw}, it is easy to find a constant $c_2>0$ such that
\begin{align*}&\W_2((\tilde{P}_t^\mu)^\ast\mu_0,\Phi(\mu))^2\leq c_2\W_2(\mu_0,\Phi(\mu))^2, \ \ t\in[0,2], \mu_0\in\scr P_2(\R^{2d}), \mu\in\scr P_2(\R^{2d}).
\end{align*}
So, we have
\begin{align}\label{der}\nonumber&\W_2((\tilde{P}_t^\mu)^\ast\mu_0,\Phi(\mu))^2\\
&\leq \tilde{c}_2\e^{-2\lambda t}\W_2(\mu_0,\Phi(\mu))^2, \ \ t\geq 0, \mu_0\in\scr P_2(\R^{2d}), \mu\in\hat{\scr P}_{2,N}(\R^{2d}).
\end{align}
Again by \eqref{Lipvw}  and Gronwall's inequality, we obtain
\begin{align}\label{samin}\W_2((\tilde{P}_t^\mu)^\ast\Phi(\nu),(\tilde{P}_t^\nu)^\ast\Phi(\nu))^2\leq \e^{(2+2K_M)t}\int_0^tK_I^2\W_2(\mu,\nu)^2\d s,\ \ \mu,\nu\in\scr P_2(\R^{2d}).
\end{align}
Combining \eqref{der} for $\mu_0=\Phi(\nu)$ with \eqref{samin}, we have
\begin{align*}&\W_2(\Phi(\mu),\Phi(\nu)) \leq \sqrt{\tilde{c}_2}\e^{-\lambda t}\W_2(\Phi(\mu),\Phi(\nu))+\sqrt{\e^{(2+2K_M)t}t}K_I\W_2(\mu,\nu), \ \ \mu,\nu\in\hat{\scr P}_{2,N}(\R^d).
\end{align*}
Let $\eta_0$ be given as in Theorem \ref{fixdi} and set
$$\eta_1=\left(\inf_{t>\frac{\log {\sqrt{\tilde{c}_2}}}{\lambda}} \frac{\sqrt{\e^{(2+2K_M)t}t}}{1-\sqrt{\tilde{c}_2}\e^{-\lambda t}}\right)^{-1}.$$
So, when $K_I<\min(\eta_0,\eta_1)$,
we can derive from the Banach fixed theorem that \eqref{Degth1} has a unique invariant probability measure $\bar{\mu}\in\hat{\scr P}_{2,N}(\R^{2d})$.

Finally, we prove the exponential ergodicity in $\W_2$ and relative entropy. Similar to \eqref{samin}, we get
$$\W_2(P_t^\ast\mu_0,(\tilde{P}_t^{\bar{\mu}})^\ast\mu_0)^2\leq \e^{(2+2K_M)t}\int_0^tK_I^2\W_2(P_s^\ast\mu_0,\bar{\mu})^2\d s,\ \ \mu_0\in\scr P_2(\R^{2d}).
$$
This together with \eqref{der} for $\mu=\bar{\mu}$ and Gronwall's inequality yields
\begin{align*}
&\W_2(P_t^\ast\mu_0,\bar{\mu})^2\\
&\leq \left(\int_0^t2\tilde{c}_2\e^{-2\lambda s}2\e^{(2+2K_M)s}K_I^2\e^{2\e^{(2+2K_M)t}K_I^2(t-s)}\d s +2\tilde{c}_2\e^{-2\lambda t}\right)\W_2(\mu_0,\bar{\mu})^2.
\end{align*}
So, by the same argument to derive \eqref{CMYke}, when $K_I$ is small enough, we can choose $\hat{t}>0$ and $\alpha\in(0,1)$ such that
\begin{align*}\W_2(P_{\hat{t}}^\ast\mu_0,\bar{\mu})\leq \alpha\W_2(\mu_0,\bar{\mu}), \ \ \mu_0\in\scr P_2(\R^{2d}).
\end{align*}
By the semigroup property $P_t^\ast P_s^\ast=P_{s+t}^\ast$, we get \eqref{w3e}. Finally, by \eqref{lohar}, \eqref{w3e} and the Talagrand inequality \eqref{Talht} with $\Phi(\mu)=\bar{\mu}$, the proof is completed by the same argument as \eqref{entco}.
\end{proof}
\section{Appendix}
\subsection{Well-posedness}
In this section, we give a result on the well-posedness of \eqref{E1} and \eqref{E2} under {\bf(A)}.
\begin{thm}\label{wel} Under the assumption {\bf(A)} we have the following.
	\begin{enumerate}
		\item [$(1)$] Equation \eqref{E2} is well-posed for any $\mu_\cdot\in C([0,\infty),\scr P_2(\R^d))$ and for any $T>s$, there exists a constant $C(T)>0$ such that
		\begin{align}\label{mof}
			\E\sup_{r\in[s,T]}|X_{s,r}^\mu|^2<C(T)\left(1+\E|X_{s,s}^\mu|^2+\int_s^T\mu_r(|\cdot|^2)\d r\right).
		\end{align}
		\item [$(2)$]
Equation \eqref{E1} is well-posed in $\scr P_2(\R^d)$ and for any $T>0$, there exists a constant $C(T)>0$ such that
	\begin{align}\label{mof13}\E\sup_{s\in[0,T]}|X_s|^2\leq C(T)\left(1+\E|X_0|^2\right).
	\end{align}
	\end{enumerate}
\end{thm}
\begin{proof}
	(1) Fix $T>0$ and $\mu_\cdot\in C([0,\infty);\scr P_2(\R^d))$. Without loss of generality, we show the assertion for $s=0$, i.e. we consider
	\beq\label{EMU} \d X_t= b_t(X_t, \mu_t)\d t+  \si_t(X_t,\mu_t)\d W_t,\ \  t\in [0,T].
\end{equation}
We first note that {\bf(A)} implies
\begin{align}\label{Mon13}2\<b_t(x,\mu_t)-b_t(y,\mu_t),x-y\>+\|\sigma_t(x,\mu_t)-\sigma_t(y,\mu_t)\|_{HS}^2\leq K_1|x-y|^2,
\end{align}
\begin{align}\label{inner}
	\nonumber 2\<b_t(x,\mu_t),x\>
	&\leq 2\<b_t(x,\mu_t)-b_t(0,\delta_0),x\>+2\<b_t(0,\delta_0),x\>\\
	&\leq K_1|x|^2+K_I\mu_t(|\cdot|^2)+2|b_t(0,\delta_0)||x|,
\end{align}
and
\begin{align}\label{HSQ}
	&\|\sigma_t(x,\mu_t)\|_{HS}^2\leq K_1(1+|x|^2+\mu_t(|\cdot|^2)).
\end{align}
Since $b$ is bounded on bounded sets of $[0,\infty)\times\R^d\times\scr P_2(\R^d)$ and $\mu_\cdot\in C([0,\infty);\scr P_2(\R^d))$, we derive from \eqref{HSQ} that for any $R>0$,
\begin{align*}
	\int_0^T\sup_{|x|\leq R}\left(\|\sigma_t(x,\mu_t)\|_{HS}^2+|b_t(x,\mu_t)|\right)\d t
	&\leq T K_1\left(1+R^2+\sup_{t\in[0,T]}\mu_t(|\cdot|^2)\right)\\
	&+T\sup_{t\in[0,T],|x|\leq R,\gamma(|\cdot|^2)\leq \sup_{t\in[0,T]}\mu_t(|\cdot|^2)}|b_t(x,\gamma)|<\infty.
\end{align*}
Combining this with the continuity of $b_t(\cdot,\mu_t), \sigma_t(\cdot,\mu_t)$ due to {\bf(A)}, \eqref{Mon13}, \eqref{inner} and \eqref{HSQ}, we may apply \cite[Theorem 3.1.1]{PR} to conclude that \eqref{EMU} is strongly well-posed.

In order to show \eqref{mof} apply  It\^{o}'s formula to \eqref{EMU} and find
\begin{align}\label{dty}
	\d |X_t|^2=2\<b_t(X_t,\mu_t),X_t\>\d t+\|\sigma_t(X_t,\mu_t)\|_{HS}^2\d t+2\<\si_t(X_t,\mu_t)\d W_t, X_t\>.
\end{align}
Now we define a sequence of stopping times $\tau_{n}:=T\wedge\inf\{t\geq 0, |X_t|\geq n\},\ \ n\geq 1$ where by convention $\inf \emptyset=\infty$. By the Burkholder-Davis-Gundy inequality, we deduce that
\begin{align*}
	&\E\sup_{s\in[0,t]}\left|\int_0^{s\wedge \tau_n}2\<\si_r(X_r,\mu_r)\d W_r, X_r\>\right|\\
	&\leq C_0\E\left|\int_0^{t\wedge \tau_n}(1+|X_r|^2+\mu_r(|\cdot|^2))|X_r|^2\d r\right|^{\frac{1}{2}}\\
	&\leq \frac{1}{2}\E\sup_{s\in[0,t\wedge\tau_n]}|X_s|^2+C_1\E\int_0^{t\wedge \tau_n}(1+|X_r|^2+\mu_r(|\cdot|^2))\d r
\end{align*}
This together with \eqref{dty}, \eqref{inner}, \eqref{HSQ}, $\sup_{t\in[0,T]}|b_t(0,\delta_0)|<\infty$ and Gronwall's inequality implies
\begin{align}\label{tau}
	\E\sup_{s\in[0,t\wedge\tau_n]}|X_s|^2\leq \e^{C_2 t}\left(\E|X_0|^2+C_2\int_0^t(1+\mu_r(|\cdot|^2))\d r\right),\ \ t\in[0,T].
\end{align}
As a result, it follows that
\begin{align*}
	\P(\tau_n\leq  T)\leq \P\left(\sup_{s\in[0,T\wedge\tau_n]}|X_s|^2\geq n^2\right)&\leq \frac{\E\sup_{s\in[0,T\wedge\tau_n]}|X_s|^2}{n^2}\\
	&\leq \frac{\e^{C_2 T}\left[\E|X_0|^2+C_2\int_0^T(1+\mu_r(|\cdot|^2))\d r\right]}{n^2}.
\end{align*}
Letting $n\to\infty$, we have
\begin{align}\label{litau}
	\lim_{n\to\infty}\P(\tau_n\leq  T)=\P(\lim_{n\to\infty}\tau_n\leq  T)=0.
\end{align}
Combining this with \eqref{tau} and Fatou's lemma, we obtain \eqref{mof}.

(2) In order to show (2),
for any $\gamma\in\scr P_2(\R^d)$ and $\mu_\cdot\in C([0,\infty);\scr P_2(\R^d))$, define  $$\Phi_t^\gamma(\mu)=\L_{X_t},\ \ t\in[0,T]$$
for $X_t$ solving \eqref{EMU} with $\L_{X_0}=\gamma$.
Note that by \eqref{mof} and the dominated convergence theorem, we get
\begin{align}\label{const}\lim_{s\to t}\E|X_s-X_t|^2=0,\ \ t\in[0,T].
\end{align}

It follows from \eqref{const} that $\Phi^\gamma$ is a map on $C([0,T];\scr P_2(\R^d))$.
Set $\scr C_T^\gamma=\{\mu\in C([0,T];\scr P_2(\R^d)):\mu_0=\gamma\}$. Since \eqref{mof13} follows from \eqref{mof} with $\mu_r=\L_{X_r}$ and Gronwall's inequality, it remains to prove that for large enough $\lambda>0$, $\Phi^\gamma$ is a contractive map under  $$\rho_\lambda (\mu^1,\mu^2):=
\sup_{t\in[0,T]}\e^{-\lambda t}\W_2(\mu^1_t,\mu^2_t),\ \ \mu^1,\mu^2\in C([0,T];\scr P_2(\R^d)). $$

Let $\L_{X_0^1}=\gamma$, $\mu^1,\mu^2\in \scr C_T^\gamma$ and $X_t^i, i=1,2 $ solve \eqref{EMU} with $\mu^i,i=1,2$ replacing $\mu$ and $X_0^1=X_0^2$.
By It\^{o}'s formula and \eqref{Mon}, we have
\begin{align*}
	\d |X_t^1-X_t^2|^2&\leq K_1|X_t^1-X_t^2|^2\d t+K_I\W_2(\mu_t^1,\mu_t^2)^2\d t\\
	&+  2\<[\si_t(X^1_t,\mu_t^1)-\si_t(X^2_t,\mu_t^2)]\d W_t,X_t^1-X_t^2\>.
\end{align*}
Let $\zeta_n=T\wedge\inf\{t\geq 0:|X_t^1|\geq n\}\wedge\inf\{t\geq 0:|X_t^2|\geq n\}, n\geq 1$. So, we arrive at
\begin{align*}
	\E |X_{t\wedge\zeta_n}^1-X_{t\wedge\zeta_n}^2|^2&\leq \int_0^{t}\left(K_1\E|X_s^1-X_s^2|^2\d s+K_I\W_2(\mu_s^1,\mu_s^2)^2\right)\d s,\ \ t\in[0,T].
\end{align*}
By \eqref{litau} for $X_t^i,i=1,2$ replacing $X_t$, we have $\P$-a.s. $\lim_{n\to\infty}\zeta_n=T$. Letting $n\to\infty$, Fatou's lemma implies
\begin{align*}
	\E |X_{t}^1-X_{t}^2|^2&\leq \int_0^{t}\left(K_1\E|X_s^1-X_s^2|^2\d s+K_I\W_2(\mu_s^1,\mu_s^2)^2\right)\d s.
\end{align*}
Gronwall's inequality yields
\begin{align*}
	\W_2(\Phi_t^\gamma(\mu^1),\Phi_t^\gamma(\mu^2))^2\leq \E |X_{t}^1-X_{t}^2|^2&\leq K_I\e^{K_1 T}\int_0^{t}\W_2(\mu_s^1,\mu_s^2)^2\d s.
\end{align*}
So, we conclude that
\begin{align*}
	\sup_{t\in[0,T]}\e^{-2\lambda t}\W_2(\Phi_t^\gamma(\mu^1),\Phi_t^\gamma(\mu^2))^2\leq \frac{K_I\e^{K_1 T}}{2\lambda }\sup_{t\in[0,T]}\e^{-2\lambda t}\W_2(\mu^1_t,\mu^2_t)^2.
\end{align*}
Taking $\frac{K_I\e^{K_1 T}}{2\lambda_0 }=\frac{1}{4}$, we get
\begin{align*}
	\sup_{t\in[0,T]}\e^{-\lambda_0 t}\W_2(\Phi_t^\gamma(\mu^1),\Phi_t^\gamma(\mu^2))\leq \frac{1}{2}\sup_{t\in[0,T]}\e^{-\lambda_0 t}\W_2(\mu^1_t,\mu^2_t).
\end{align*}
The proof is completed by the Banach fixed point theorem.
\end{proof}
\subsection{Methods of approximation}
In this part, we provide two kinds of approximation, one is the Yosida approximation in Lemma \ref{yap} and the other is the approximation by mollifier in Lemma \ref{aky}.

Let $\hat{b}:[0,\infty)\times\R^d\to\R^d,$ $\hat{\sigma}:[0,\infty)\to\R^d\otimes\R^n$ be measurable. Consider
\begin{align}\label{mut}\d \hat{X}_{s,t}=\hat{b}_t(\hat{X}_{s,t})\d t+\hat{\sigma}_t\d W_t,\ \ t\geq s\geq 0.
\end{align}

\subsubsection{Yosida approximation}

We make the following assumption.
\begin{enumerate}
\item[{\bf(A1)}] $\hat{b}$ and $\hat{\sigma}$ are locally bounded and for any $t\geq0$, $b_t$ is continuous on $\R^d$. There exists a locally bounded function $K:[0,\infty)\to[0,\infty)$ such that
\begin{align}\label{monbt}
2\<\hat{b}_t(x_1)-\hat{b}_t(x_2),x_1-x_2\>\leq K(t)|x_1-x_2|^2,\ \ x_1,x_2\in\R^d,t\geq 0.
\end{align}
\end{enumerate}
Let
    $
        \tilde{b}_t(x):=\hat{b}_t(x)
        -\frac12 K(t)x,\ \ x\in\R^{d}.
    $
    Then for any $t\geq 0$, $\tilde{b}_t$ is also continuous and
    $$\<\tilde{b}_t(x)-\tilde{b}_t(y),x-y\>\leq 0,\ \ t\geq 0,x,y\in\R^d.$$
    For any $n\geq 1$, let
    $$
        \tilde{b}^{(n)}_t(x)
        :=n\left[
        \left(\operatorname{id}
        -\frac{1}{n}\tilde{b}_t\right)^{-1}(x)-x
        \right],\quad x\in\R^d,
    $$
    where $\mathrm{id}$ is the identity map on $\R^d$. By \cite[Proposition D.11]{GZ}(see also \cite[Section 2]{GRW}),
for any $n\geq 1$, $|\tilde{b}^{(n)}|\leq|\tilde{b}|$, $\lim_{n\to\infty}\tilde{b}^{(n)}=\tilde{b}$,
and
$$\<\tilde{b}_t^{(n)}(x)-\tilde{b}_t^{(n)}(\tilde{x}),x-\tilde{x}\>\leq 0,\ \ |\tilde{b}^{(n)}_t(x)-\tilde{b}^{(n)}_t(\tilde{x})|\leq 2n|x-\tilde{x}|,\ \ t\geq 0, x,\tilde{x}\in \R^d, n\geq 1. $$
Moreover, let
    \begin{align}\label{bhn}\hat{b}_t^{(n)}(x):=\tilde{b}^{(n)}_t(x)+\frac12 K(t)x, \ \ t\geq 0,x\in\R^d, n\geq 1.
    \end{align} Then $\hat{b}^{(n)}$ satisfies
\begin{align}\label{bnl}|\hat{b}^{(n)}_t(x)-\hat{b}^{(n)}_t(\tilde{x})|\leq \left(2n+\frac{1}{2}K(t)\right)|x-\tilde{x}|,\ \ t\geq 0, x,\tilde{x}\in \R^d,n\geq 1,
\end{align}
\begin{align}\label{bt0}|\hat{b}^{(n)}_t(0)|=|\tilde{b}^{(n)}_t(0)|\leq |\tilde{b}_t(0)|=|\hat{b}_t(0)|,\ \ t\geq 0,n\geq 1,
\end{align}
    and
    \begin{equation}\label{jg4dv7}
        2\<\hat{b}^{(n)}_t(x)-\hat{b}^{(n)}_t(y),x-y\>\leq K(t)|x-y|^2,\ \ n\geq 1, x,y\in\R^d, t\geq 0.
    \end{equation}
     Let $(\hat{X}_{s,t}^{(n)})_{0\leq s\leq t}$ solve \eqref{mut} with $\hat{b}^{(n)}$ replacing $\hat{b}$ and $\hat{X}_{s,s}^{(n)}=\hat{X}_{s,s}$. Then we have the following result.
\begin{lem}\label{yap} Assume {\bf(A1)}. Then for any $T>0$, there exists a constant $C(T)>0$ such that
\begin{align}\label{sup13}
\sup_{n\geq 1}\E\sup_{t\in[s,T]}|\hat{X}_{s,t}^{(n)}|^2\leq C(T)(1+\E|\hat{X}_{s,s}|^2),
\end{align}
and for any $t\geq s$, $\lim_{n\to\infty} |\hat{X}_{s,t}^{(n)}-\hat{X}_{s,t}|^2=0.$
\end{lem}
\begin{proof} Repeating the proof of \eqref{mof}, we derive \eqref{sup13} by \eqref{jg4dv7} and \eqref{bt0}. Fix $t\geq s$.
    By It\^{o}'s formula and {\bf (A1)}, we derive
    \begin{align*}
        &\d|\hat{X}_{s,r}^{(n)}-\hat{X}_{s,r}|^2\leq (K(r)+1)|\hat{X}_{s,r}^{(n)}-\hat{X}_{s,r}|^2\d r+\big|\hat{b}^{(n)}_r(\hat{X}_{s,r})-\hat{b}_r(\hat{X}_{s,r})\big|^2\,\d r.
    \end{align*}
 Then Gronwall's inequality implies that
\begin{align*}
|X_{s,t}^{(n)}-X_{s,t}|^2&\leq\e^{(\sup_{r\in[s,t]}K(r)+1)(t-s)}\int_s^{t}\big|\tilde{b}^{(n)}_r(\hat{X}_{s,r}) -\tilde{b}_r(\hat{X}_{s,r})\big|^2\,\d r.
\end{align*}
    Since $\tilde{b}$ is locally bounded with $|\tilde{b}^{(n)}|\leq |\tilde{b}|$ and $\lim_{n\to\infty}\tilde{b}^{(n)}=\tilde{b}$,     the dominated convergence theorem yields that for any $n\geq 1$,
    $$\lim_{n\to\infty} |\hat{X}_{s,t}^{(n)}-\hat{X}_{s,t}|^2=0.$$
    The proof is now finished.
    \end{proof}
\subsubsection{Approximation by mollifier}
Next, we construct an approximation by mollifier. We make the following assumptions.
\begin{enumerate}
\item[{\bf(A2)}] There exists a locally bounded function $L:[0,\infty)\to[0,\infty)$ such that
\begin{align}
\label{btL}|\hat{b}_t(x)-\hat{b}_t(\bar{x})|\leq L_t|x-\bar{x}|,\ \ t\geq 0
\end{align}
and
\begin{align}\label{bsg}
|\hat{b}_t(0)|+\|\hat{\sigma}_t\|\leq L_t.
\end{align}
\end{enumerate}
 Let $0\leq \rho\in C_0^\infty(\R^d)$  satisfying $\int_{ \R^d}\rho(x)\d x=1$ and for any $m\geq 1$, set
$$\rho^m(x)=m^{d}\rho(mx)$$
and
\begin{align}\label{bmt}
\hat{b}^{m}_t(x)=\int_{\R^d}\hat{b}_t(x-y)\rho^m(y)\d y.
\end{align}
Then it is not difficult to see from \eqref{btL} that
\begin{align}\label{uniform}\sup_{x\in\R^d}|\hat{b}^{m}_t(x)-\hat{b}_t(x)|\leq L_t\int_{\R^d}|y|\rho^m(y)\d y=L_t\frac{1}{m}\int_{\R^d}|y|\rho(y)\d y,
\end{align}
    \begin{align}\label{uniform1}
        \nonumber&|\hat{b}^{m}_t(x)-\hat{b}^{m}_t(\tilde{x})|\\
        &\leq \int_{\R^d}|\hat{b}_t(x-y)-\hat{b}_t(\tilde{x}-y)|\rho^m(y)\d y\leq L_t|x-\tilde{x}|,\ \ m\geq 1, x,\tilde{x}\in\R^d, t\geq 0,
    \end{align}
\begin{align}\label{growthli}|\hat{b}^{m}_t(x)|\leq |\hat{b}^{m}_t(x)-\hat{b}_t(x)|+|\hat{b}_t(x)|\leq L_t\frac{1}{m}\int_{\R^d}|y|\rho(y)\d y+L_t|x|+|\hat{b}_t(0)|,
\end{align}
and for any $m\geq 1, i\geq 1$, we can find a locally bounded function $\ell^{(m),i}:[0,\infty)\to[0,\infty)$ only depending on $L_t,t\geq 0$ such that
\begin{align}\label{ktl13}\|\nabla^i\hat{b}^{m}_t\|_\infty\leq \ell^{(m),i}_t, \ \ t\geq 0.
\end{align}
Moreover,  if \eqref{monbt} holds,
then one has
    \begin{align}\label{bmm}
\nonumber        2\<\hat{b}^{m}_t(x)-\hat{b}^{m}_t(\tilde{x}),x-\tilde{x}\>&=2\int_{\R^d}\<\hat{b}_t(x-y)-\hat{b}_t(\tilde{x}-y),x-y-(\tilde{x}-y)\>\rho^m(y)\d y\\
      &\leq K(t)|x-\tilde{x}|^2,\ \ m\geq 1, x,\tilde{x}\in\R^d, t\geq 0.
    \end{align}
    This means that the approximation by mollifier keeps the monotone property.
Let $\hat{X}_{s,t}^{m}$ solve \eqref{mut} with $\hat{b}^{m}$ replacing $\hat{b}$ and $\hat{X}_{s,s}^{m}=\hat{X}_{s,s}$.
\begin{lem}\label{aky} Assume {\bf(A2)}. Then for any $T>0$, there exists a constant $C(T)>0$ such that
\begin{align}\label{supmh}
\sup_{m\geq 1}\E\sup_{t\in[s,T]}|\hat{X}_{s,t}^{m}|^2\leq C(T)(1+\E|\hat{X}_{s,s}|^2),
\end{align}
and for any $t\geq s$, $\lim_{m\to\infty} |\hat{X}_{s,t}^{m}-\hat{X}_{s,t}|=0.$
\end{lem}
\begin{proof} By \eqref{growthli} and \eqref{bsg}, one can repeat the proof of \eqref{mof} to derive \eqref{supmh}.
It follows from \eqref{uniform} and \eqref{uniform1} that
    \begin{align*}
        |\hat{X}_{s,t}^{m}-\hat{X}_{s,t}|&\leq \int_s^tL_r|\hat{X}_{s,r}^{m}-\hat{X}_{s,r}|\d r+\int_s^t\big|\hat{b}^{m}_r(\hat{X}_{s,r})-\hat{b}_r(\hat{X}_{s,r})\big|\,\d r\\
        &\leq \int_s^tL_r|\hat{X}_{s,r}^{m}-\hat{X}_{s,r}|\d r+\frac{1}{m}\int_{\R^d}|y|\rho(y)\d y\int_s^tL_r\d r.
    \end{align*}
The proof is completed by first using Gronwall's lemma and then letting $m$ go to infinity.
\end{proof}

\end{document}